\documentclass[12pt,a4paper]{article}%
\usepackage{amsmath}
\usepackage[left=30mm,top=25mm,right=30mm,bottom=30mm]{geometry}
\usepackage{amsfonts}
\usepackage{amssymb}
\usepackage{bbm}
\usepackage{amsthm}
\usepackage{soul}
\usepackage{color}
\usepackage[utf8]{inputenc}
\usepackage[pdftex]{graphicx}
\usepackage{epstopdf}
\usepackage{enumerate}
\usepackage{hyperref}

\usepackage{tikz} 
\usetikzlibrary{arrows}
\usetikzlibrary{calc}
\usetikzlibrary{shapes}
\usetikzlibrary{positioning}
\usetikzlibrary{shapes.geometric}
\usetikzlibrary{arrows.meta}

\providecommand{\U}[1]{\protect\rule{.1in}{.1in}}

\renewenvironment{proof}[1][Proof]{\noindent\textbf{#1.} }{\ \rule{0.5em}{0.5em}}
\newenvironment{acknowledgement}{\textbf{Acknowledgement.}}{}

\newtheorem{theorem}{Theorem}

\newtheorem{corollary}[theorem]{Corollary}

\newtheorem{lemma}{Lemma}[section]
\newtheorem{proposition}[lemma]{Proposition}

\theoremstyle{definition}
\newtheorem{definition}[lemma]{Definition}
\newtheorem*{definition*}{Definition}
\newtheorem{remark}[lemma]{Remark}
\newtheorem{example}[lemma]{Example}

\renewcommand{\P}{\mathbb{P}}
\newcommand{\E}{\mathbb{E}}
\newcommand{\N}{\mathbb{N}}
\newcommand{\Z}{\mathbb{Z}}

\newcommand{\G}{\mathcal{G}}
\renewcommand{\H}{\mathbb{H}}
\newcommand{\R}{\mathbb{R}}

\newcommand{\Stab}{\mathrm{Stab}}
\newcommand{\Isom}{\mathrm{Isom}}
\newcommand{\core}{\mathrm{core}}
\newcommand{\acc}{\mathrm{acc}}
\newcommand{\RD}{\mathbb{RD}}
\newcommand{\BRD}{\mathbb{BRD}}
\newcommand{\Ends}{\mathrm{Ends}}

\newcommand{\GS}{\mathcal{G}_{\Sigma}}
\newcommand{\Cut}{\mathrm{Cut}}
\newcommand{\righte}{\overrightarrow{e}}
\newcommand{\lefte}{\overleftarrow{e}}
\newcommand{\rightf}{\overrightarrow{f}}

\newcommand{\rightp}{\overrightarrow{p}}
\newcommand{\leftp}{\overleftarrow{p}}

\title{A full characterization of invariant embeddability of unimodular planar graphs}
\author{Ádám Timár and László Márton Tóth}
\date{}

\begin{document}

\maketitle
\let\thefootnote\relax

\begin{abstract}
When can a unimodular random planar graph be drawn in the Euclidean or the hyperbolic plane in a way that the distribution of the random drawing is isometry-invariant? 
This question was answered for one-ended unimodular graphs in \cite{benjamini2019invariant}, using the fact that such graphs automatically have locally finite (simply connected)
drawings into the plane. For the case of graphs with multiple ends the question was left open. We revisit Halin's graph theoretic characterization of graphs that have a locally finite embedding into the plane.
Then we prove that such unimodular random graphs do have a locally finite invariant embedding into the Euclidean or the hyperbolic plane, depending on whether the graph is amenable or not.
\end{abstract}

\section{Introduction}

Consider a random planar map embedded in the Euclidean or hyperbolic plane $M$ with an isometry-invariant distribution. Simple examples include a lattice shifted by a suitable random isometry or a Voronoi tessellation coming from some invariant point process in $M$ \cite{moller2012lectures, benjamini2001percolation}.
If the expected number of vertices in a unit area is finite, then one can condition on having a vertex in a fixed point $0$ of $M$, and define the Palm version, which is hence a random {\it rooted} graph embedded in the plane. This rooted graph is always unimodular (in the sense defined by Aldous and Lyons in \cite{aldous2007processes}), even if the embedding is also taken into account as a decoration (marking) of the rooted graph. (For Euclidean spaces a proof using the terminology of this present paper can be found in \cite{benjamini2019invariant}, but this is a basic result in the theory of point processes, where unimodularity is replaced by point-stationarity \cite{heveling2005characterization, chiu2013stochastic}.) It is natural to ask whether the converse is true: {\it does every unimodular random planar graph $(G,o)$ have an isometry-invariant embedding into $M$, with a suitable choice of $M$ as the Euclidean or hyperbolic plane?} By a unimodular random planar 
graph (URPG) we mean a unimodular random graph that is planar (has a proper embedding in the plane) almost surely. 

To break the above question into two parts, when does $(G,o)$ have a unimodular embedding into $M$ (where a {\it unimodular embedding}, to be formally defined later, is one where the embedding as a decoration of the graph makes it a unimodular decorated graph), with the root $o$ is embedded at the origin $0 \in M$? And can such an embedding be used to define an isometry-invariant embedding? The present paper fully answers both questions. In \cite{benjamini2019invariant}, the question of unimodular embeddings of {\it one-ended} URPG's was answered, always in the affirmative, with amenable $G$'s embeddable into the Euclidean plane and nonamenable $G$'s into the hyperbolic plane. Our main contribution here is the solution of the question of unimodular embeddings of URPG's with two or infinitely many ends, whenever it is possible. We also construct isometry invariant embeddings from unimodular ones, which settles the question for invariant embeddings.

The problem in the present setup is somewhat more delicate than in the one-ended case, because the existence of an embedding depends on the graph structure. We mention that 2-ended unimodular random graphs are always invariantly amenable, while infinitely ended ones are always invariantly non-amenable. (See the next section for the definitions of invariant amenability/non-amenability, which will nevertheless not be needed until Section \ref{section:noembedding}.) This will imply that in the former case embeddings into the Euclidean plane are meant, while in the latter case embeddings into the hyperbolic plane.

For the special situation when $G$ is transitive, the expression of ``an isometry-invariant embedding of $G$'' makes sense right away: we want the random embedded graph to be isomorphic to this fixed graph almost surely. (This is how the term is used in \cite{timar2017nonamenable}.) However, when the given rooted graph $(G,o)$ is in fact random, one has to assign a root to the embedded graph to be able to ask for an almost sure rooted isomorphism. This is why it was essential to introduce the Palm version above, and for that definition to work, we need the set of embedded vertices, as a point process, to have finite intensity. If the vertices of a random embedded graph have some accumulation point, then the corresponding invariant point process has infinite intensity. In particular, if every embedding of a given URPG has some accumulation point, then the above question of isometry-invariant embeddings is not defined for that graph. Let us mention that there exist random graphs embedded in the Euclidean plane in an invariant way such that it is not possible to assign to them a root so that the resulting rooted graph is unimodular. See Example \ref{meglepo}. This shows that defining an invariant embedding of a URPG through the Palm version and hence the requirement of finite intensity, is not the artifact of a possibly wrongly chosen definition, but seems to be the most general and natural choice. 

Say that an embedding into a surface is \emph{locally finite} if every compact set is intersected by only finitely many embedded edges.  A point of the surface is an {\it accumulation point} if all its neighborhoods intersect infinitely many embedded edges.
As we will see later, it is easy to find examples of URPG's 
that do not have a locally finite embedding into the plane, and even for those that do, it is a priori unclear whether such an embedding can be chosen to be unimodular. We are led to the following questions:
\begin{enumerate}
\item Give some combinatorial characterization of infinite locally finite graphs that have a locally finite embedding in the plane.
\item Characterize those unimodular random planar graphs where there is a {\it unimodular} embedding into $\R^2$ or $\H^2$ with the above property, and construct such an embedding. Similarly, find an isometry-invariant embedding.
\end{enumerate}
We will fully answer both questions. For the first one, a  characterization is given not only for the plane, but other orientable surfaces as well (Theorem \ref{theorem:simply_connected_comb_embedding}). The planar case itself (Theorem \ref{starminor}) was already characterized by Halin \cite{halin1966haufungspunktfreien}. These results are analogues of Kuratowski's theorem with minors replaced by ``minors that can use infinity''.

For the second problem, having a locally finite (not necessarily unimodular) embedding 
will turn out to be equivalent to having a unimodular {\it combinatorial} embedding in the plane with no accumulation points (a notion yet to be defined in the context of combinatorial embeddings), see Theorem \ref{theorem:unimodular_combinatorial_embedding_1_acc}. These are further equivalent to having a locally finite unimodular embedding into the Euclidean plane or the hyperbolic plane, which will be equivalent to having an isometry-invariant embedding (Theorem \ref{thm:invariant_and_unimodular_embedding}).
Another variant of our result says that having a one-ended URPG as a unimodular supergraph is equivalent to having a (locally finite) invariant embedding into the Euclidean or hyperbolic plane (combine Theorem \ref{supergraph} with Theorem \ref{thm:invariant_and_unimodular_embedding} for one direction, and Proposition \ref{prop:triangularization} for the other).

When there is a positive answer to the second problem, our starting point is the construction in \cite{timar2023unimodular} for a unimodular {\it combinatorial} embedding of a URPG. A combinatorial version of ``locally finite'' has to be introduced, and we modify the construction of \cite{timar2023unimodular} to make sure that the unimodular combinatorial embedding is locally finite whenever $G$ has a locally finite embedding at all (Theorem \ref{theorem:unimodular_combinatorial_embedding_1_acc}). After obtaining the unimodular locally finite combinatorial embedding, one has to realize it as an actual locally finite embedding into the plane. At this point, the metric imposed on the plane needs to be clarified, to make sense of what is meant by the embedding to be unimodular. The two natural choices are the Euclidean and the hyperbolic metric. Then a dichotomy result holds (Theorem \ref{thm:invariant_and_unimodular_embedding}),
similarly to \cite{benjamini2019invariant}.

\subsection{Our main results} 
As before, let $M$ denote either $\R^2$ or $\H^2$. A \emph{drawing} of a graph $G$ on $M$ is a locally finite embedding $\iota$ of $G$ to $M$, viewed up to isometries of $M$. A \emph{rooted drawing} is a drawing together with a distinguished vertex. We denote by $\RD (M)$ the space of rooted drawings of locally finite graphs on $M$, with a suitable topology given by ``local closeness'' in a neighborhood of the root. A probability measure $\mu$ on $\RD(M)$ is \emph{unimodular}, if it satisfies the appropriate Mass Transport Principle. We give the detailed definitions in Subsection \ref{subsec:unimodular_embeddings}. 
 
A \emph{unimodular embedding} of a URPG $(G,o)$ into $M$ is a unimodular measure $\mu$ on $\RD(M)$, such that forgetting the embedding of a $\mu$-random element 
(and taking the rooted-isomorphism class of the resulting rooted graph)
gives $(G,o)$ in distribution. This was first defined in \cite{angel2018hyperbolic}; a related notion of unimodular discrete spaces is studied in \cite{baccelli2021unimodular}.

Let $\mathfrak{G}$ denote a random  embedded unrooted graph on $M$ with $\Isom(M)$-invariant distribution. (Both the graph and its embedding can be random.) The \emph{intensity} of $\mathfrak{G}$ is $p(\mathfrak{G})=\E [ |V(\mathfrak{G}) \cap B|]$, where $B \subset M$ is an arbitrary measurable set of area $1$. When the intensity is finite, one
can define the \emph{Palm version} $\mathfrak{G}^*$ of $\mathfrak{G}$, by conditioning on $0 \in V(\mathfrak{G}))$, where $0$ is the origin in $M$. By standard theory of point processes this makes sense (\cite[Chapter 9]{last2017lectures}), and $\mathfrak{G}^*$ defines a random rooted drawing, with the distinguished vertex at $0$. 

An \emph{invariant embedding} of a URPG $(G,o)$ into $M$ is an invariant random locally finite embedded unrooted graph $\mathfrak{G}$ such that the Palm version of $\mathfrak{G}$ is a unimodular embedding of $(G,o)$. So when invariant embeddings of URPG's are considered, finite intensity is inherent.
 
It is clear that in order to find either a unimodular or an invariant embedding of a URPG without accumulation points, we need to assume that it is supported on graphs that have a locally finite planar embedding. We show that under this assumption a URPG has an invariant or unimodular embedding of finite, positive intensity into either the Euclidean or the hyperbolic plane, depending on whether it is (invariantly) amenable or not. (So far intensity only has a meaning for invariant embeddings, but the intensity of a unimodular embedding will also be defined, in Subsection \ref{subsection:invariant_from_unimodular}.) 
 We say that a URPG $(G,o)$ has {\it finite expected degree} if the expected degree of $o$ is finite.

\begin{theorem} \label{thm:invariant_and_unimodular_embedding} 
Let $(G,o)$ be a URPG with finite expected degree, and assume that $G$ has a locally finite embedding into the plane with probability one. Then $(G,o)$ has an invariant embedding and also a unimodular embedding with positive intensity and no accumulation point into 
\begin{itemize}
\item the Euclidean plane if and only if $(G,o)$ is invariantly amenable;
\item the hyperbolic plane if and only if $(G,o)$ is invariantly non-amenable.
\end{itemize}
\end{theorem}

%

Theorem \ref{thm:invariant_and_unimodular_embedding} is an extension of similar results in \cite{benjamini2019invariant}. Here we do not require our URPG to be one-ended, replacing that assumption with having a locally finite embedding. Furthermore, the invariant embedding in the non-amenable case was not investigated in \cite{benjamini2019invariant} (and the method of obtaining an invariant embedding from a unimodular one in the hyperbolic plane is not straightforward). 

To better understand the assumptions of Theorem \ref{thm:invariant_and_unimodular_embedding}, one can ask for a graph theoretic characterization of infinite locally finite graphs that have a locally finite embedding in the plane. Such a characterization is given by Halin in \cite{halin1966haufungspunktfreien}, see \cite[Theorem 1.1]{bonnington2003graphs} for a statement in English. Suppose $G$ is infinite. Loosely speaking, we say that a finite graph $H$ is a $*$-minor of $G$ if we can find $H$ as a minor of $G$ with one of the vertices of $H$ being ``at infinity'' in $G$. 
See Definition \ref{definition:star_minor}. 

\begin{theorem}[Halin]\label{starminor}
An infinite locally finite planar graph has a locally finite planar embedding if and only if it does not have $K_{3,3}$ or $K_5$ as a $*$-minor.
\end{theorem}

In fact we proved a version of this statement for embedding graphs in orientable surfaces, and only found Halin's theorem upon review of the present paper. See Theorem \ref{theorem:simply_connected_comb_embedding} for our more general result.

For completeness, let us state the $*$-minor characterizations of invariant/unimodular embeddability as a separate theorem. This is straightforward
from Theorem \ref{starminor} and Theorem \ref{thm:invariant_and_unimodular_embedding}.

\begin{theorem}\label{amen_and_nonamen_emb}
Let $(G,o)$ be an invariantly amenable (respectively, invariantly non-amenable) URPG with finite expected degree. Then 
\begin{itemize}
\item $(G,o)$ has a unimodular embedding of finite positive intensity into $\R^2$ (respectively, into $\H^2$) if and only if it does not have $K_{3,3}$ or $K_5$ as a $*$-minor. 
\item It has an invariant embedding of positive intensity into $\R^2$ (respectively, into $\H^2$) if and only if it does not have $K_{3,3}$ or $K_5$ as a $*$-minor. 
\item It has no unimodular embedding of positive intensity and no invariant embedding into $\H^2$ (respectively, into $\R^2$).
\end{itemize}
\end{theorem}


As mentioned at the beginning, most of Theorem \ref{thm:invariant_and_unimodular_embedding} was already proved in \cite{benjamini2019invariant} for the case when $G$ is almost surely one-ended. Our next theorem shows that one-endedness is, in a sense, the real reason for a graph to be unimodularly/invariantly embeddable in the plane. 

\begin{theorem}\label{supergraph}
Let $(G, o)$ be a URPG with finite expected degree. Then the following are equivalent.
\begin{enumerate}
\item[(1)] $G$ has a locally finite planar embedding almost surely, and hence Theorem \ref{thm:invariant_and_unimodular_embedding} applies.
\item[(2)] $G$ has a unimodular one-ended planar supergraph.
\end{enumerate}
\end{theorem}

\begin{proof}
It is a straightforward corollary of Proposition \ref{prop:triangularization} below.
\end{proof}

Our constructions for the embeddings in Theorem \ref{thm:invariant_and_unimodular_embedding} will follow the strategy adopted in \cite{benjamini2019invariant}, once the following is established. 

\begin{proposition} \label{prop:triangularization}
Let $(G, o)$ be a URPG with finite expected degree that has a locally finite planar embedding almost surely. Then there is a decorated URPG $(G^+, o^+; S)$, with a connected, positive density subgraph $S$ of $G^+$,  such that
$(G^+,o^+)$ is a planar triangulation of finite expected degree, and $G^+$ has one end. (Positive density means $\P[o^+ \in S]>0$.) Furthermore, conditioned on $o^+\in S$, $(S,o^+)$ is distributed as $(G,o)$, and $(G^+, o^+)$ is invariantly amenable if and only if $(G,o)$ is.
\end{proposition}

The definition of unimodularity for decorated graphs is introduced in Subsection \ref{subsec:unimodularity}. Proposition \ref{prop:triangularization} is a generalization of Theorem 2.2 in \cite{benjamini2019invariant}. It will follow from Theorem \ref{theorem:unimodular_combinatorial_embedding_1_acc} presented below, which is the main novelty and most laborsome result of this paper.

Already in the one-ended case a key ingredient of the proof is to construct \emph{combinatorial embeddings} for URPG's in a unimodular way, see \cite{timar2023unimodular}. Combinatorial embeddings are thoroughly discussed in Section \ref{section:prelim}. For now it is enough to know that a combinatorial embedding is a decoration of the vertices by finite sets of labels, where each label describes the cyclic order of edges around the vertex in an embedding, and together the labels encode the embedding of the graph up to homeomorphisms. 
We will also define the number $\acc(G, \pi)$ of accumulation points of a graph $G$ with a combinatorial embedding $\pi$, see Definition \ref{def:comb_acc}. The difficulty in our case is coming from the fact that the unimodular combinatorial embedding constructed in \cite{timar2023unimodular} can have $\acc(G, \pi) > 1$ for graphs with more than one end. Here we give a different, more careful construction.

\begin{theorem} \label{theorem:unimodular_combinatorial_embedding_1_acc}
Let $(G,o)$ be a URPG with finite expected degree, which has a locally finite planar embedding almost surely. Then $(G,o)$ has a unimodular random combinatorial embedding $\pi$ with $\acc(G, \pi) \leq 1$ almost surely. 
\end{theorem}

The paper is structured as follows. Definitions are introduced in Section \ref{section:prelim}, together with some notation. Section \ref{section:surface_graphs} contains a 
graph theoretic characterization of infinite graphs that have some embedding into a given surface with only
one accumulation point, yielding a proof of Theorem \ref{starminor} as a special case.

The proof of our main Theorem \ref{thm:invariant_and_unimodular_embedding}, presented in Section \ref{section:inv_uni_emb}, can be broken down into parts. In the Euclidean case we construct an invariant embedding first, which can be trivially used to define a unimodular embedding (similarly to \cite{benjamini2019invariant}). In the hyperbolic case we are able to construct a unimodular embedding first. Obtaining an invariant embedding from the unimodular one is the nontrivial direction, we prove the necessary  Theorem \ref{theorem:invariant_from_unimodular} in Subsection \ref{subsection:invariant_from_unimodular}. 
Our first step towards a unimodular embedding in the hyperbolic case is to find unimodular {\it combinatorial} embeddings into the plane with one combinatorial accumulation point. That is, we prove Theorem \ref{theorem:unimodular_combinatorial_embedding_1_acc} in Section \ref{section:unimod_combin}. This is the lengthiest and most laborsome section of our paper.

A unimodular combinatorial embedding as 
above can then be used to define an actual unimodular embedding of positive intensity into $\H^2$ when the graph is invariantly non-amenable. This was known before, via circle packings, and we briefly summarize the method in Section \ref{section:inv_uni_emb}. Here we also prove Proposition \ref{prop:triangularization} as a consequence of Theorem \ref{theorem:unimodular_combinatorial_embedding_1_acc}, and the construction parts of Theorem \ref{thm:invariant_and_unimodular_embedding} using Proposition \ref{prop:triangularization} and Theorem \ref{theorem:invariant_from_unimodular}. Finally in Section \ref{section:noembedding} we prove the other direction, the necessary conditions of our embeddability theorems.

\section{Preliminaries} \label{section:prelim}

Our focus in this paper is planar graphs and unimodularity. Nevertheless, throughout Sections \ref{section:prelim} and \ref{section:surface_graphs} we will keep our setup more general, and consider embeddings of arbitrary graphs into orientable surfaces. The proof of Theorem \ref{theorem:unimodular_combinatorial_embedding_1_acc} and as a consequence Theorem \ref{thm:invariant_and_unimodular_embedding} will not make use of any of this general setup.

All graphs studied are assumed to be connected with at least 2 vertices. Disconnected graphs only appear when we delete subgraphs to study how the components of the remainder behave. $\mathcal{C}_G(x)$ denotes the connected component of the vertex $x$ in the graph $G$.

Let $H' \subseteq H$ be a subgraph. The graph $H \setminus H'$ is obtained by removing the vertices of $H'$ and all adjacent edges from $H$. For a vertex $v \in V(H)$ let $N_H(v)$ denote the neighbors of $v$ in $H$.

To remove a set of edges $E \subseteq E(H)$ we also write $H \setminus E$. In this case all vertices of $H$ are kept, only the edges are deleted.

\subsection{Unimodular random graphs}
\label{subsec:unimodularity}

A \emph{random rooted graph} is a probability measure $\nu$ on $\mathcal{G}_{\bullet}$, the space of rooted, connected, locally finite graphs, considered up to rooted isomorphisms. The set $\mathcal{G}_{\bullet}$ is a locally compact Polish space, the topology is defined by the rooted distance $d_r$, where
\[d_r\big((G_1,o_1),(G_2,o_2) \big) = \frac{1}{2^{k+1}}, \textrm{ where } k=\sup \big\{i \in \N ~\big|~ B_{G_1}(o_1,i) \cong B_{G_2}(o_2,i)\big\}.\]

We say $\nu$ is \emph{unimodular}, if it satisfies the Mass Transport Principle \cite{aldous2007processes}. That is, for any measurable function $f:\mathcal{G}_{\bullet \bullet} \to [0,\infty)$ on the space $\mathcal{G}_{\bullet \bullet}$ of connected, locally finite birooted graphs (up to birooted isomorphisms) we have

\[\int_{(G,o)} \sum_{u \in V(G)} f(G,o, u) \ d \nu(G,o) = \int_{(G,o)} \sum_{u \in V(G)} f(G,u, o) \ d \nu(G,o).\]

The definition can be repeated when vertices or edges of the graph are decorated by elements of a complete separable metric space. See \cite{aldous2007processes} or \cite[Chapter 18.3]{lovasz2012large} for details. Usually we write $(G,o)$ for a unimodular random rooted graph, meaning implicitly that its distribution is a unimodular measure. When a decoration, like a function $f$ on edges or vertices, or a subset $U$ of edges or vertices is present, we write $(G,o;f)$ or $(G,o;U)$.
An equivalent formulation of unimodularity is by involution invariance. Assume $(G,o)$ has finite expected degree, and let $(G',o')$ result from $(G,o)$ after biassing by the degree of $o$. Let $(G',o',o'')$ denote the random birooted graph obtained by taking a uniform random neighbor $o''$ of $o'$ in $G'$. The random rooted graph $(G,o)$ is unimodular if and only if the distribution of $(G',o',o'')$ is the same as the distribution of $(G',o'',o')$. The same equivalence holds for decorated graphs.

We say a unimodular random graph $(G,o)$ is \emph{invariantly amenable}, if for every $\varepsilon >0$ there is a random subset $U \subseteq V(G)$ such that $(G,o;U)$ is unimodular, every component of $G \setminus U$ is finite, and $\P[o \in U] < \varepsilon$. (See \cite{aldous2007processes} for equivalent definitions.) Otherwise it is called \emph{invariantly non-amenable}. For simplicity from now on we will refer to these as amenable and non-amenable. As we only consider unimodular random graphs in this paper, this should cause no confusion. 

\subsection{Embeddings of arbitrary graphs}
For the rest of this Section \ref{section:prelim} let $G$ be a connected, locally finitely graph and $\Sigma$ a closed, connected, and orientable surface. An \emph{embedding} of $G$ into $\Sigma$ is a map $\iota$ that maps vertices of $G$ to distinct points of $\Sigma$ and edges $e$ with endpoints $x$ and $y$ to arcs between $\iota(x)$ and $\iota(y)$ such that no inner point of these arcs is contained in another arc. We denote by $\iota(G)$ the union of all these vertices and arcs in $\Sigma$. We call the connected components of $\Sigma \setminus \iota(G)$ the \emph{topological faces} of $\iota(G)$. 
We say the embedding is \emph{amicable}, if for any finite subgraph $F$ of $G$, the topological faces of $\iota(F)$ are the interiors of compact orientable surfaces with boundary. All embeddings considered in this paper are assumed to be amicable. For planar embeddings in this paper we also assume amicability, by which we mean amicability of the embedding as viewed to be mapping into the one-point compactification $S^2$.

A few remarks are in order:

\begin{itemize}
\item We do not require the the embedding to be \emph{cellular}, i.e.\ the topological faces of $F$ do not have to be disks. In the planar case, which is our main interest, amicable embeddings are automatically cellular. In general however, topological faces of $F$ are assumed to be of the form $\Sigma' \setminus (D_1 \cup \ldots \cup D_k)$, where $\Sigma'$ is a closed orientable surface and the $D_j \subseteq \Sigma'$ are disjoint closed disks. This makes our setup somewhat different from the ones discussed in \cite{lando2013graphs} and \cite{mohar2001graphs}. The reason for our choice is that we want to make sure that graphs that can be embedded into a fixed surface form a minor-closed family. This is not the case if we require all topological faces to be disks. The only place where this remark applies is Theorem \ref{theorem:simply_connected_comb_embedding}, where we go beyond the planar setup.

\item In the definition of amicability, we do not impose our topological requirement on topological faces of infinite subgraphs, because we want to allow examples with ``infinite faces'', like the natural embedding of $\Z$ into $\R^2$. The complement of $\Z$ on $S^2$ is connected, but not homeomorphic to an open disk. The requirement on the finite subgraphs will give us the right generality and enough control.

\item When considering planar embeddings of infinite graphs on $S^2$, the embedding is of course not locally finite anymore. 
Nevetheless, such locally finite planar embeddings are in one-to-one correspondence with embeddings into $S^2$ with exactly one accumulation point (which point is not in the image of any edge or vertex by the embedding). We will reserve the term  \emph{simply connected}, for such embeddings into $S^2$.  So ``simply connected'' and ``locally finite planar'' embeddings are the same, but we will use the two terminologies to indicate the space we are embedding into.
\end{itemize}

The \emph{graph boundary} of a topological face of $G$ (with respect to $\iota$) is the subgraph that is mapped onto the topological boundary of the face by $\iota$. So while a topological face of $G$ is a domain in $\Sigma$, the graph boundary of a topological face in $G$ is a subgraph of $G$. (The term ``simply connected'', frequently used in the literature, indicates that the union of the closures of the topological faces with finite graph boundary is a simply connected topological space.)

\subsection{Combinatorial embeddings} \label{subsection:combinatorial_embeddings}
In this subsection we introduce some notions from the theory of graphs on surfaces. See \cite[Chapter 1.3.3]{lando2013graphs}, \cite[Chapter 3.1]{mohar2001graphs} or \cite{beineke2009topics} for more details on the topic, but keep in mind that our setup is slightly different, because we do not assume embeddings to be cellular.

Fix an embedding $\iota$ of $G$ into $\Sigma$. For each edge $e$ of $G$ introduce two oriented edges $\righte$ and $\lefte$, or \emph{darts} in opposite directions. We denote by $s(\righte)$ and $t(\righte)$ the source and terminal vertex of a dart. When the edge is a loop it still gives two opposing darts, but in this case the source and terminus are the same.

The orientation of the surface defines a cyclic permutation $\pi_v$ at every vertex $v \in V(G)$ of the darts leaving that vertex. The collection of all these permutations $(\pi_v)_{v \in V(G)}$ is the \emph{combinatorial embedding} of $G$ corresponding to $\iota$. In general such a collection of cyclic permutations is called a combinatorial embedding, if there is some $\iota$ that realizes it.

We go on to explore how this combinatorial data can be used to reconstruct the embedding when $G$ is finite.

Given a graph $G$ with a combinatorial embedding $(\pi_v)_{v \in V(G)}$ into $\Sigma$ we can define the \emph{combinatorial faces} of $G$. A combinatorial face is a cyclically ordered tuple of darts $(\righte_1, \ldots, \righte_n)$ with $t(\righte_i)=s(\righte_{i+1})=v_i$ and such that the dart $\righte_{i+1}$ is the successor of the dart $\lefte_{i}$ according to the permutation $\pi_{v_i}$. We understand $i+1$ cyclically, so $\righte_{n+1}=\righte_{1}$. 
We can think of a combinatorial face as a returning walk along oriented edges, always continuing on the unique succeeding dart at every vertex. Note that a combinatorial face might contain both darts belonging to the same edge. When $G$ is finite, for every dart $\righte$ there is a unique combinatorial face that contains it. We will revisit the infinite case in subsection~\ref{subsection:infinite_blocks}.

When $\Sigma=S^2$ the embedding is automatically cellular, so each combinatorial face is the graph boundary of a topological face, and these topological faces are all homeomorphic to the open disk. So gluing disks to the combinatorial faces of the combinatorial embedding $\pi$ reconstructs an embedding that realizes $\pi$. In particular if two embeddings $\iota_1$ and $\iota_2$ define the same combinatorial embedding, then they are the same up to a homeomorphism of $S^2$. That is, there exists a homeomorphism $\varphi: S^2 \to S^2$ such that $\iota_1 = \varphi \circ \iota_2$. 

When $\Sigma$ is some other orientable surface the picture is more complicated. The cyclic graph $C_n$ can be embedded in the torus $T^2$ two different ways that are not the same up to homeomorphisms, yet define the same (and only) combinatorial embedding of $C_n$. (Neither of these embeddings is cellular.) To be able to reconstruct an embedding from the combinatorial data we further need to record the genus of the topological faces of $G$, as well as the number of closed disks removed (see Example \ref{example:cycle_on_torus}). Each topological face $F_i$ is of the form $\Sigma_i \setminus (D^i_1 \cup \ldots \cup D^i_{k_i})$, where $\Sigma_i$ is an orientable surface and the $D^i_j \subseteq \Sigma_i$ are disjoint closed disks ($k_i \geq 1$). 

The graph boundary $\partial F_i$ is a union of combinatorial faces of $G$. Each combinatorial face is part of the graph boundary of exactly one topological face. So we get a partition of the combinatorial faces $F_c(G,\pi)$ of $(G,\pi)$ into sets $P_i \subseteq F_c(G,\pi)$. We have $P_i \cap P_j = \emptyset$ for $i\neq j$ and $\cup_i P_i = F_c(G,\pi)$. To each subset $P_i$ we record the genus $g_i$ of $\Sigma_i$. Note that $g_i \leq g(\Sigma)$. 

To summarize, we collected the data $\pi=(\pi_v)_{v \in V(G)}$, the partition $(P_i)$ of $F_c(G,\pi)$ and for each $P_i$ the number $g_i$. This information is sufficient to reconstruct the embedding $\iota$ by gluing $\Sigma_i \setminus (D^i_1 \cup \ldots \cup D^i_{k_i})$ to $G$ along the combinatorial faces collected in $P_i$, where $\Sigma_i$ is the surface of genus $g_i$. Here $k_i=|P_i|$. If two embeddings give the same combinatorial data, they are the same up to homeomorphisms. Note that the correspondence between the combinatorial faces in $P_i$ and the $D^i_j$ need not be recorded for the reconstruction of the embedding, as any bijective mapping gives the same embedding up to homeomorphisms.

\begin{corollary} \label{corollary:finitely_many_embeddings}
Let the finite graph $G$ and the orientable surface $\Sigma$ be fixed. There are finitely many embeddings of $G$ into $\Sigma$ up to homeomorphisms.
\end{corollary}

\begin{proof}
There are finitely many choices for the combinatorial data $\pi$, $(P_i)$ and $(g_i)$, with $g_i \leq g(\Sigma)$. 
\end{proof}

\begin{example} \label{example:cycle_on_torus}
For the cycle $C_n$ let $\pi$ denote its only combinatorial embedding (which is an involution of the two outgoing darts at every vertex). Let $E$ and $E'$ denote the two combinatorial faces. When embedding $C_n$ into $T^2$ the two distinct embeddings produce the combinatorial data $(\pi, P_1=\{E,F\}, g_1=0)$ and $(\pi, P_1=\{E\},P_2=\{F\}, g_1=0, g_2=1)$.
\end{example}

\begin{remark}
Not all possible choices of $\pi$, $(P_i)$ and $(g_i)$ with $g_i \leq g(\Sigma)$ will give an embedding into $\Sigma$. Even in the cellular case (meaning each $P_i$ is a singleton and all $g_i$ are $0$), the genus of the surface created by the gluing procedure might vary depending on $\pi$. See \cite[Chapter 3]{beineke2009topics} where the enumeration of cellular maps into surfaces of different genus is discussed.
\end{remark}

\subsection{Unimodular embeddings} \label{subsec:unimodular_embeddings}

We study invariant and unimodular embeddings in the planar case, that is graphs embeddable into $\Sigma = S^2$. In the previous subsections we considered embeddings up to homeomorphisms. As we mentioned before, the difference between amenable and non-amenable URPG's becomes apparent when we specify the metric on the punctured sphere, and consider embeddings up to isometries.

As in the Introduction, let $M$ denote either $\R^2$ or $\H^2$. We denote by $\mathbb{G}(M)$ the space of graphs amicably embedded into $M$. Elements of $\mathbb{G}(M)$ are denoted by typewriter font letters like ${\tt G}$. $\mathbb{G}(M)$ carries a standard Borel structure by requiring all edge- and vertex-counting functions over Borel subsets of $M$ to be measurable. The group $\Gamma = \Isom(M)$ acts on $\mathbb{G}(M)$ measurably by shifting the vertices and edges. The shift of ${\tt G} \in \mathbb{G}(M)$ by $\varphi \in \Gamma$ is denoted $\varphi.{\tt G}$. We denote by $\mathbb{G}_0(M)$ the subspace of embedded graphs that have a vertex at the origin $0 \in M$; this is again a standard Borel space. We denote by $\mathbb{G_{\bullet}(M)}=\{({\tt G}, {\tt o}) \mid {\tt G} \in \mathbb{G}(M), {\tt o} \in V({\tt G})\}$ the set of rooted embedded graphs.

A \emph{drawing} of a graph on $M$ is \emph{an equivalence class} of locally finite, amicably embedded graphs, where two such embeddings are equivalent if they are the same up to an element of $\Gamma$. The set of such drawings is denoted $\mathbb{D}(M)$. The equivalence class of an embedded graph ${\tt G} \in \mathbb{G}(M)$ is denoted by $[{\tt G}]$. 

A \emph{rooted drawing} is a drawing together with a distinguished vertex. More precisely, the space of rooted drawings $\mathbb{RD}(M)$ is $\mathbb{G_{\bullet}(M)}/\Gamma$. The equivalence class of $({\tt G}, {\tt o}) \in \mathbb{G_{\bullet}(M)}$ is denoted by $[{\tt G}, {\tt o}]$. Equivalently, a rooted drawing can be thought of as a rooted graph embedded into $M$ with the root at the origin $0 \in M$, up to isometries of $M$ fixing $0$. Therefore $\mathbb{RD}(M)$ is in fact in bijection with $\mathbb{G}_0(M) / \Stab_{\Gamma}(0)$. We use this correspondence to define the Borel structure: as $\RD(M)$ is a factor of a standard Borel space by a compact subgroup it is itself standard Borel. The embedded graph ${\tt G} \in \mathbb{G}_0(M)$ represents the class $[{\tt G},0]$.



We say a probability measure $\mu$ on $\RD(M)$ is \emph{unimodular}, if it satisfies the appropriate Mass Transport Principle. Namely, one defines the space $\BRD(M)$ of \emph{birooted drawings} of graphs on $M$ similarly to $\RD(M)$, the equivalence class of a $(\mathtt{G},\mathtt{u},\mathtt{v})$ is denoted by $[\mathtt{G},\mathtt{u},\mathtt{v}]$. A \emph{payment function} is a measurable function $f: \BRD(M) \to [0,\infty)$. We say $\mu$ is unimodular if the expected income of the root of a $\mu$-random rooted drawing is equal to the expected outpay, for any payment function $f$:

\[\int_{\RD(M)} \sum_{\mathtt{u} \in V(\mathtt{G})} f([\mathtt{G},\mathtt{o}, \mathtt{u}]) \ d \mu([\mathtt{G},\mathtt{o}]) = \int_{\RD(M)} \sum_{\mathtt{u} \in V(\mathtt{G})} f([\mathtt{G}, \mathtt{u}, \mathtt{o}]) \ d \mu([\mathtt{G},\mathtt{o}]).\]

In the formula we have to choose a representative $(\mathtt{G},\mathtt{o})$ of $[\mathtt{G},\mathtt{o}]$ to compute the sum over $V(\mathtt{G})$, even though we integrate over $[\mathtt{G},\mathtt{o}]$. This is still well-defined because the value of the sum does not depend on the choice of representative. This problem of having to choose representatives for formulas to make sense is a recurring phenomenon in this section; our intention is to keep the notation reasonably clean and intuitive. We will point this out whenever it might cause confusion.

When defining unimodular and invariant embeddings of URPG's, we are interested in random variables with values in the spaces $\mathbb{G}(M)$ and $\RD(M)$. We will use gothic letters like $\mathfrak{G}$ to denote them. A \emph{unimodular embedding} of a URPG $(G,o)$ into $M$ is a random rooted drawing $[\mathfrak{G}, \mathfrak{o}]$ such that its distribution $\mu$ is unimodular, and  forgetting the embedding of any representative $(\mathfrak{G},\mathfrak{o})$ gives $(G,o)$ in distribution.

\begin{remark}
It is standard to take random rooted graphs themselves as equivalence classes of rooted graphs up to rooted isomorphism. (As we implicitly did in subsection \ref{subsec:unimodularity}.) One could ask why we never highlighted this in our notation, yet keep emphasizing the difference between classes and representatives when it comes to rooted drawings. The reason is that, in the case of drawings, equivalence is up to an isometry of the ambient space. This is much less intuitive from the name ``rooted drawing'' than the usual identification of isomorphic rooted graphs.  
\end{remark}

\subsection{Invariant embeddings} \label{subsec:invariant_embeddings}

In a moment we will define invariant embeddings through unimodularity of their Palm versions. Before that we would like to note that properties of Palm versions of invariant point processes are thoroughly studied in the literature, but conditions on point processes that inherently characterize them to be the Palm version of some invariant point process are much less frequent. Nevertheless, in the Euclidean setting such techniques have been elaborated (see \cite{heveling2005characterization}, and also \cite[Proposition II.12]{neveu1977processus}). The theory was also extended to Abelian locally compact groups (see \cite{last2009invariant}), and then to homogeneous spaces of general locally compact groups, in particular covering the hyperbolic case. The work of Last \cite{last2010stationary} covers the theory introduced in the present and consequent subsections in this generality, while much of what we need here is already present in \cite{rother1990palm}.

Recall that $\Gamma = \Isom(M)$, and $\mathbb{G}(M)$ is the space of amicably embedded graphs. A \emph{random embedded unrooted graph} $\mathfrak{G}$ is a random variable with values in $\mathbb{G}(M)$. When $\mathfrak{G}$ is $\Gamma$-invariant in distribution, the \emph{intensity} is $p(\mathfrak{G})=\E [ |V(\mathfrak{G}) \cap B|]$, where $B \subset M$ is an arbitrary measurable set of area $1$.  The value $p(\mathfrak{G})$ does not depend on the choice of $B$ because of the $\Gamma$-invariance. Notice that the intensity cannot be $0$, but it can be $\infty$. For $\mathfrak{G}$ with finite intensity, let $\mathfrak{G}^*$ denote its Palm version. Then the distribution of $[\mathfrak{G}^*,0]$ is a unimodular measure on $\RD(M)$, see e.g.\ Example 9.5 in \cite{aldous2007processes}. We say $\mathfrak{G}$ is an \emph{invariant embedding} of the URPG $(G,o)$ into $M$, if $[\mathfrak{G}^*,0]$ is a unimodular embedding of $(G,o)$. To sum up, an invariant embedding of a given URPG is a $\Gamma$-invariant random embedded unrooted graph, whose Palm version rooted at 0 gives back the URPG after forgetting the embedding.

Notice that we assume finite intensity of $\mathfrak{G}$ to be able to choose a root in the graph, giving rise to a unimodular random rooted graph: the choice is provided by the Palm version. The next example indicates that the requirement of finite intensity is necessary.

\begin{example}\label{meglepo}
Consider an invariant point process of finite intensity in the Euclidean plane; let $\omega$ be the corresponding random point set. For every $x\in\omega$, let $p_{x,i}$ be a uniformly chosen point from the ball of radius $2^{-i}$ around $x$. Define the Voronoi tessallation for the set $\{p_{x,i},x\in\omega,i=1,2\ldots\}$. Connect two points if their Voronoi cells are adjacent. The resulting graph has countably infinitely many ends, hence it is not unimodular with any choice of root. (Like transitive graphs, unimodular random rooted graphs have 0, 1, 2 or uncountably many ends, see \cite[Proposition 6.10]{aldous2007processes}.) On the other hand it is a random graph drawn in the plane with an isometry-invariant distribution.
\end{example}

\subsection{Invariant embeddings from unimodular ones}
\label{subsection:invariant_from_unimodular}

Assume as before that $\mathfrak{G}$ is a $\Gamma$-invariant embedded unrooted graph with $p(\mathfrak{G})<\infty$, and therefore $[\mathfrak{G}^*, 0]$ is unimodular. If $\lambda$ denotes the area in $M$ and ${\rm Vor}(\mathfrak{G}^*,0)$ the Voronoi cell of $0$, then 

\begin{equation} \label{eqn:voronoi}
E_{\mathfrak{G}^*}\big[\lambda\big({\rm Vor}(\mathfrak{G}^*,0)\big)\big] = \frac{1}{p(\mathfrak{G})}.
\end{equation}
Equation (\ref{eqn:voronoi}) is very intuitive, and a basic corollary of the \emph{Voronoi inversion formula}, a basic tool in the theory of point processes, at least in the Euclidean case (\cite[Section 9.4]{last2017lectures}). For homogeneous spaces \cite{last2010stationary} covers the theory.

Inspired by (\ref{eqn:voronoi}) we introduce the \emph{intensity} of a unimodular embedding $[\mathfrak{G_0},\mathfrak{o}]$ as 
\[p\big([\mathfrak{G_0},\mathfrak{o}]\big) = \E_{[\mathfrak{G_0},\mathfrak{o}]}\big[\lambda\big({\rm Vor}(\mathfrak{G_0},\mathfrak{o})\big)\big]^{-1}.\]

The reason for this less direct definition is that a priori fixing a measurable subset $B$ of the space is not possible (as it was in the definition of intensity of an invariant embedding), since in this context everything is understood up to isometries. The value $\lambda\big({\rm Vor}(\mathfrak{G_0},\mathfrak{o})\big)$ is well defined, it does not depend on the choice of representative of $[\mathfrak{G_0},\mathfrak{o}]$. 
Notice that we allow the intensity of a unimodular embedding to be $0$ (when $\E_{[\mathfrak{G_0},\mathfrak{o}]}\big[\lambda\big({\rm Vor}(\mathfrak{G_0},\mathfrak{o})\big)\big]$ is infinite), but not $\infty$. 

Moreover, the intensity of $\mathfrak{G}$ can be measured using any factor allocation (to be defined next), not just the Voronoi. An \emph{allocation scheme} is a method of partitioning the space $M$ (up to measure 0) into measurable pieces, given a discrete subset $\omega \subset M$, such that the measurable pieces are bijectively associated to the points in $\omega$. One also assumes that the partition depends on $\omega$ in a measurable way. For example in the Voronoi allocation scheme each point of $\omega$ is associated to its Voronoi cell. We denote the piece associated to a point $v \in \omega$ by $\Psi_{\omega}(v)$.

The allocation is a \emph{factor allocation} if it is equivariant, i.e., for any isometry $g$ of $M$ we have $\Psi_{g. \omega}(g. v) = g. \Psi_{\omega}(v)$. If the random point set $\omega$ is defined as the embedded vertices of some URPG by a unimodular embedding, which is by definition only defined up to rooted isometries of $M$, the property of equivariance guarantees that such an allocation scheme provides a well-defined allocation to the points of these $\omega$ almost surely. Hence factor allocations are automatically well-defined for unimodular embedded graphs (even the 0 intensity ones). We denote the scheme itself by $\Psi$. 

If $\mathfrak{G}$ is $\Gamma$-invariant random, then for any factor allocation $\Psi$ we have the analogue of (\ref{eqn:voronoi}), 
\[\E_{\mathfrak{G}^*}\big[\lambda\big(\Psi_{V(\mathfrak{G}^*)}(0)\big)\big] = \frac{1}{p(\mathfrak{G})}.\]
The following theorem allows us to reverse this correspondence. 

\begin{theorem}[Theorem 7.1 in \cite{last2010stationary} and Theorem 3 in \cite{rother1990palm}] \label{theorem:invariant_from_unimodular}
Let $[\mathfrak{G_0}, \mathfrak{o}]$ be a unimodular random rooted drawing, and $\Psi$ a factor allocation scheme such that the expected area of the cell of $\mathfrak{o}$ is finite, that is $\mathbb{E}_{[\mathfrak{G_0},\mathfrak{o}]}\left[ \lambda \big( \Psi_{V(\mathfrak{G_0})}(\mathfrak{o}) \big) \right] < \infty$. Then there is a $\Gamma$-invariant random  embedded unrooted graph $\mathfrak{G}$ of finite intensity such that $[\mathfrak{G}^*,0]$ is the same as $[\mathfrak{G_0}, \mathfrak{o}]$ in distribution. 

In particular if a URPG $(G,o)$ has a unimodular embedding of positive intensity, it also has an invariant embedding of finite intensity. 
\end{theorem}

For the sake of completeness and full agreement with our terminology and notation, we include the proof of Theorem~\ref{theorem:invariant_from_unimodular} in the Appendix.


The Palm version of an invariant process determines the process, therefore as a corollary of Theorem \ref{theorem:invariant_from_unimodular} one gets the following.

\begin{corollary}\label{cor:intensity}
Let $[\mathfrak{G_0}, \mathfrak{o}]$ be a unimodular random rooted drawing into $M$. Then for any factor allocation schemes $\Psi$ and $\Psi'$ we have 
\[\E_{[\mathfrak{G_0}, \mathfrak{o}]}\left[\lambda\big(\Psi_{V(\mathfrak{G_0})}(\mathfrak{o})\big)\right] = \E_{[\mathfrak{G_0}, \mathfrak{o}]}\left[\lambda\big(\Psi'_{V(\mathfrak{G_0})}(\mathfrak{o})\big)\right].\]
\end{corollary}

\begin{remark}
Again, there is a slight abuse of notation in the corollary above. A unimodular random rooted drawing only gives an equivalence class, so the cell of the root is not defined as a subset of $M$. Nevertheless, as both allocation schemes are factors, the area of the cell of the root is well-defined.
\end{remark}

\section{Surface graphs with 1 accumulation point} \label{section:surface_graphs}

Fix a compact orientable surface $\Sigma$. Let $\GS$ denote the family of locally finite, connected graphs that can be embedded in $\Sigma$. Let $\GS^f$ denote the family of finite graphs in $\GS$. As $\GS^f$ is minor closed, by the Robertson-Seymour graph minor theorem (see \cite{robertson2004graph} and references therein, or Chapter 12.5 in \cite{diestel2017graph}) it can be characterized by a finite set of forbidden minors $M_{\Sigma}=\{H_1, \ldots, H_n\}$.

\[\GS^f = \{G \textrm{ finite, connected} \mid H_i \textrm{ is not a minor of } G, \ \forall \ 1\leq i \leq n\}.\]

This characterization extends to $\GS$ almost automatically:

\begin{lemma} \label{lemma:embedding_from_exhaustion}
Let $G$ be a locally finite graph, and $G_1 \subseteq G_2 \subseteq \ldots$ an exhaustion of $G$ by finite graphs. If all $G_i$ can be embedded in $\Sigma$, then $G$ can also be embedded in $\Sigma$.
\end{lemma}

\begin{proof}
By Corollary \ref{corollary:finitely_many_embeddings} each $G_i$ has finitely many embeddings into $\Sigma$ up to homeomorphisms. Hence, by Konig's lemma, we can find a coherent sequence of embeddings of the $G_i$ into $\Sigma$. That is, embeddings $\iota_i$ of $G_i$ and homeomorphisms $\varphi_i: \Sigma \to \Sigma$ such that $\iota_i = \varphi_i \circ \iota_{i+1} \vert_{G_i}$.

Then we define an embedding $\iota$ of $G$ by setting $\iota \vert_{G_i} = \varphi_1 \circ \ldots \circ \varphi_{i-1} \circ \iota_i$. This is well defined since it stabilizes for every edge, and indeed an embedding, because any unwanted intersection of arcs would be witnessed in some $G_i$.
\end{proof}

\vspace{0.3cm}

We can now characterize $\GS$ using minors. The possibly infinite, locally finite graph $H$ is said to be a minor of the locally finite graph $G$ if we can contract (possibly infinitely many) edges and delete edges and isolated vertices of $G$ to obtain $H$.

\begin{remark}
Note that contracting infinitely many edges of $G$ can result in graphs that are not locally finite. We however restrict our definition to the case where the minor $H$ is also locally finite.
When $H$ is actually finite the situation is more straightforward, $H$ is a minor of $G$ if and only if it is a minor of a finite subgraph of $G$ in the traditional sense. 
\end{remark}

\begin{lemma}
The family $\GS$ consists of all locally finite graphs not having a minor in $M_{\Sigma}$.
\[\GS = \{G \textrm{ locally finite, connected} \mid H \textrm{ is not a minor of } G, \ \forall H \in  M_{\Sigma}\}.\]
\end{lemma}

\begin{proof}
Let $G$ be a locally finite graph, $G_1 \subseteq G_2 \subseteq \ldots$ be an exhaustion of $G$ by finite graphs. A finite graph $H$ is a minor of $G$ if and only if $H$ is a minor of some $G_i$. By Lemma \ref{lemma:embedding_from_exhaustion} $G \in \GS$ if and only if $G_i \in \GS^f$ for all $i$.
\end{proof}

\vspace{0.3cm}
We now turn to embeddings with 1 accumulation point. 

\begin{definition*}
Let $\GS^{*}$ denote the family of graphs in $\GS$ that can be embedded in $\Sigma$ with at most $1$ accumulation point. 
\end{definition*}

\begin{remark}
It is clear that $\GS^{*}$ is minor closed, yet this does not imply it can be characterized by a finite set of forbidden minors, because the Robertson-Seymour theorem applies to finite graphs, and it is an open question whether it extends to countable graphs \cite[Chapter 12.7]{diestel2017graph}.
\end{remark}

Recall that for finite graphs, $H$ is a minor of $G$ if and only if we can 
find
$|H|$-many 
pairwise disjoint subtrees of $G$ 
such that contracting the subtrees gives a supergraph of $H$. Our definition of $*$-minor follows this formulation.

\begin{definition} \label{definition:star_minor}
Let $|H|=n$. $H$ is a $*$-minor of $G$ if we can find $n-1$ finite subtrees $X_1, \ldots, X_{n-1}$ and  a subforest $X_n$ with all connected components infinite such that
\begin{enumerate}
\item The vertex sets $V(X_i)$, $1 \leq i \leq n$, partition $V(G)$.
\item \label{item:contraction_collapse} $H$ is isomorphic to a subgraph of the graph we get from $G$ by contracting all edges of the $X_i$ and identifying all the points obtained from distinct connected components of $X_n$ to a single one. 
\end{enumerate}
\end{definition}

The graph obtained in part \ref{item:contraction_collapse} might have infinitely many loops at the vertex coming from $X_n$. This makes no difference for us.  

\begin{remark}
Note that a minor of $G$ is not necessarily a $*$-minor of $G$. In fact, finite graphs have no $*$-minors at all. See Figure \ref{fig:K_3_3_star_minor} for an infinite graph with $K_{3,3}$ as a $*$-minor. The blue $*$ indicates the point of $K_{3,3}$ ``at infinity''.
\end{remark}

\begin{lemma}
$H$ is a $*$-minor of $G$ if and only if $G$ has an infinite, locally finite minor $G'$ such that 
\begin{enumerate}
\item $G'$ has a subgraph isomorphic to $H \setminus v$ for some $v \in V(H)$. By slight abuse of notation $(H \setminus v) \subset G'$.
\item For all vertices $u \in N_H(v)\subseteq (H \setminus v)$ there is an infinite path starting at $u$ in the graph $G' \setminus (H \setminus v)$.
\end{enumerate}
\end{lemma}

\begin{proof}
Contracting the finite trees in the definition gives $G'$ and vice versa. If $G' \setminus (H \setminus v)$ has some finite connected components, then we can contract them into $(H \setminus v)$ to ensure that we get a partition of the vertices. 
\end{proof}

\vspace{0.3 cm}

\begin{figure}[h]
     \centering
         \begin{tikzpicture}[node distance=1cm,
main node/.style={circle,minimum width=1.9pt, fill, inner sep=0pt,outer sep = 0pt},
red node/.style={circle,minimum width=2.2pt, fill,red, inner sep=0pt,outer sep = 0pt},
blue node/.style={circle,minimum width=2.2pt, fill,blue, inner sep=0pt,outer sep = 0pt},
square node/.style={draw,regular polygon,regular polygon sides=4,fill,inner sep=0.9pt,outer sep=0.1pt},
triangle node/.style={draw,regular polygon,regular polygon sides=3,fill,inner sep=0.8pt,outer sep=0.1pt}]   

  \node[red node] (a) at (-1,1){};
  \node[red node] (b) at (-1,0){};
  \node[red node] (c) at (-1,-1){};
  \node[blue node] (u) at (1,0.5){};
  \node[blue node] (v) at (1,-0.5){};
 
  \path [draw=gray, thin] (a) edge (u);
  \path [draw=gray, thin] (b) edge (u);
  \path [draw=gray, thin] (c) edge (u);
  \path [draw=gray, thin] (a) edge (v);
  \path [draw=gray, thin] (b) edge (v);
  \path [draw=gray, thin] (c) edge (v);

  \node[main node] (a1) at (-2,1){};
  \node[main node] (a2) at (-3,1){};
  \node[main node] (a3) at (-4,1){};

  \node[main node] (b1) at (-2,0){};
  \node[main node] (b2) at (-3,0){};
  \node[main node] (b3) at (-4,0){};

  \node[main node] (ab4) at (-5,0.5){};
  \node[main node] (ab5) at (-6,0.5){};
  \node[main node] (ab6) at (-7,0.5){};

  \node[main node] (c1) at (-2,-1){};
  \node[main node] (c2) at (-3,-1){};
  \node[main node] (c3) at (-4,-1){};
  \node[main node] (c4) at (-5,-1){};
  \node[main node] (c5) at (-6,-1){};
  \node[main node] (c6) at (-7,-1){};

  \path [draw=gray, thin] (a) edge (a1);
  \path [draw=gray, thin] (a1) edge (a2);
  \path [draw=gray, thin] (a2) edge (a3);
  \path [draw=gray, thin] (a3) edge (ab4);

  \path [draw=gray, thin] (b) edge (b1);
  \path [draw=gray, thin] (b1) edge (b2);
  \path [draw=gray, thin] (b2) edge (b3);
  \path [draw=gray, thin] (b3) edge (ab4);
  
  \path [draw=gray, thin] (ab4) edge (ab5);
  \path [draw=gray, thin] (ab5) edge (ab6);

  \path [draw=gray, thin] (c) edge (c1);
  \path [draw=gray, thin] (c1) edge (c2);
  \path [draw=gray, thin] (c2) edge (c3);
  \path [draw=gray, thin] (c3) edge (c4);
  \path [draw=gray, thin] (c4) edge (c5);
  \path [draw=gray, thin] (c5) edge (c6);

  \node (AB) at (-8,0.5) [label=center:{$\dots$}]{};
  \node (C) at (-8,-1) [label=center:{$\dots$}]{};
  \node (x) at (-9,0) [label=center:{\textcolor{blue}{$*$}}]{};

\end{tikzpicture}
        \caption{An infinite graph that has $K_{3,3}$ as a $*$-minor}
        \label{fig:K_3_3_star_minor}
\end{figure}

We are now ready to state the main result of this section. Let $\Sigma_g$ denote the orientable closed surface with genus $g$ (which is unique up to homeomorphisms). Clearly $\mathcal{G}_{\Sigma_g} \subset \mathcal{G}_{\Sigma_{g+1}}$.

\begin{theorem} \label{theorem:simply_connected_comb_embedding}
A planar graph is in $\mathcal{G}_{S^2}^{*}$ if and only if it does not have $K_{3,3}$ or $K_5$ as a $*$-minor. For a positive genus $g \geq 1$ a graph $G \in (\mathcal{G}_{\Sigma_{g}} \setminus \mathcal{G}_{\Sigma_{g-1}})$ is in $\mathcal{G}_{\Sigma_{g}}^{*}$ if and only if it has no $*$-minors in $M_{\Sigma_g}$.
\end{theorem}

\begin{remark}
Note that because of the $G \in (\mathcal{G}_{\Sigma_{g}} \setminus \mathcal{G}_{\Sigma_{g-1}})$ condition the theorem only deals with the case when an infinite graph is embedded into a surface of minimal genus possible. This assumption is indeed necessary. The planar graph $K_3 \times \mathbb{Z}$ has no embedding into the torus $T^2$ with a single accumulation point. 
Still it does not have $*$-minors in $M_{T^2}$, because all its $*$-minors are embeddable into $T^2$.
To see this, consider 
the embedding $\beta$ to the torus $[0,1)^2$ (with $[0,1)$ understood with 0 and 1 identified) that maps  $(x,n)\in V(K_3\times \mathbb{Z})$ to $(b_1(x), b_2(n))$, where $b_1$ is some bijection between $V(K_3)$ and $\{0,1/3,2/3\}$, and $b_2: \Z \to (0,1)$ is a strictly increasing function with $\lim_{n \to -\infty}b_2(n) =0$ and $\lim_{n \to -\infty}b_2(n) =1$. Edges are embedded into geodesics between their endpoints. If $H$ is a finite graph that arises as a $*$-minor of $G$, then, using Definition \ref{definition:star_minor}, we can make $\beta(X_n)$ connected by adding some broken line segments to it that are disjoint from $\cup_{e\in E(K_3\times \mathbb{Z})}\beta (e)\setminus \cup_{v\in V(K_3\times \mathbb{Z})}\beta(v)$.
This set together with the $\beta (X_i)$, $i=1,\ldots,n-1$ can be used to represent the vertices of $H$ (see the equivalent formulation of minors before Definition \ref{definition:star_minor}) for an embedding, with edges between them coming from $\beta(K_3\times \mathbb{Z})$. 
\end{remark}

\begin{remark}
Halin's original theorem \cite{halin1966haufungspunktfreien} is not stated using $*$-minors, but instead lists the 4 infinite graphs that, if minors of $G$, can be witnesses of $K_{3,3}$ or $K_{5}$ being a $*$-minor of $G$. Such a translation is theoretically possible in the higher genus case as well. If one knew the graphs in $M_{\Sigma_g}$, one could construct the finitely many infinite graphs that are minor-obstructions to $G$ being in $\mathcal{G}_{\Sigma_{g}}^{*}$. 
\end{remark}

\begin{proof}
Suppose that $G \in \mathcal{G}_{\Sigma_{g}}^{*}$, and assume towards contradiction that some $H \in M_{\Sigma_g}$ is a $*$-minor of $G$. Carrying out the minor operations on the embedded $G$ we can find an embedded $G'$ with $(H \setminus v) \subset G'$ for some $v \in V(H)$. This gives an embedding of $H \setminus v$ into $\Sigma_g$. 

We claim that with this embedding, all $u \in N_H(v)$ are on the topological boundary of the same topological face of $(H\setminus v)$, namely the one containing the accumulation point of the embedded $G'$. Indeed, all these $u$ are part of an infinite path in $G' \setminus (H \setminus v)$, and in the embedding of $G'$ all these infinite paths have to accumulate at the unique accumulation point. 

Since all the vertices of $N_H(v)$ are on the topological boundary of the same topological face, we can place a vertex $v$ anywhere on that topological face, and connect it with $N_H(v)$ without intersections. We found an embedding of $H \in M_{\Sigma_g}$ into $\Sigma_g$ which is a contradiction.

For the other implication, suppose that $G$ has no $*$-minors in $M_{\Sigma_g}$. Let $G_1 \subset G_2 \subset \ldots \subset G$ be an exhaustion of $G$ by finite graphs. By increasing each $G_i$ if necessary, we can assume that all components of $G \setminus G_i$ are infinite. We also assume that $N_G(V(G_i)) \subseteq G_{i+1}$.

For each $G_i$ let $V^{*}(G_i)$ denote the set of vertices $v \in V(G_i)$ that are adjacent to (an infinite component of) $G \setminus G_i$. By our second assumption $V^{*}(G_{i+1}) \cap V(G_i) = \emptyset$. We construct the graph $G_i^+$ by wiring together all vertices in $V^{*}(G_i)$ using an additional vertex. That is $V(G_i^+)=V(G_i) \cup \{p_i\}$ where $p_i$ is an auxiliary point. For the edges we have $E(G_i^+) = E(G_i) \cup \{(p_i,u) \mid u \in V^{*}(G_i)\}$.

The graph $G_i^+$ can be embedded in $\Sigma_g$. Indeed, if $H \in M_{\Sigma_g}$ were a minor of $G_i^+$, then H would be a $*$-minor of $G$. Equivalently, $G_i$ can be embedded in a way that all vertices in $V^{*}(G_i)$ are on the topological boundary of the same topological face $F_i$ of $G_i$. This holds for all $i$. 

For each $G_i$ there are finitely many such embeddings up to homeomorphisms. By compactness we can choose consistent embeddings for the $G_i$. That is, restricting the embedding of $G_i$ to $G_j$ ($j \leq i$) agrees with the embedding of $G_j$. The embeddings of the $G_i$ together give an embedding $\iota$ of $G$, as shown in Lemma \ref{lemma:embedding_from_exhaustion}. Since all vertices of $G_i$ connected to infinity are along the same topological face $F_i$, we will be able to make sure that there is only one accumulation point, by choosing $F_i$ such that $\lim F_i=\{x\}$ for some point $x \in \Sigma_g$.

In the case of planar graphs the intuition is very clear: each finite part gets embedded such that all vertices connected to infinity are along the outer face. We then claim that we can keep extending this embedding in a way that no accumulation point (besides infinity) is created. 

To be precise, we argue as follows. For some $i$ large enough, $G_{i}$ cannot be embedded in a surface of smaller genus, which implies that all topological faces must be disks. In particular the topological face $F_i$ is a closed disk $D_{i}$. Fix a homeomorphism between $D_i$ and the unit disk in the Euclidean plane. This allows us to metrize $D_i$ such that it has radius 1. 

By changing the embedding of $G_{i+1}$ by a homeomorphism only inside $D_i$, we can assume that all vertices in $V(G_{i+1}) \setminus V(G_{i})$ lie in the region $D_i \setminus D_{i+1}$, where $D_{i+1}$ is the disk concentric to $D_i$ with radius $1/2$. Moreover we make sure that $V^{*}(G_{i+1})$ lies on the topological boundary of $D_{i+1}$.

Repeating this for all $k\geq 0$, we achieve that $V(G_{i+k+1}) \setminus V(G_{i+k})$ is mapped in the region $D_{i+k} \setminus D_{i+k+1}$, where $D_{i+k}$ is concentric to $D_i$ with radius $1/2^k$, while making sure that $V^{*}(G_{i+k})$ lies on the topological boundary of $D_{i+k}$. Note that at each step we only change the embedding inside $D_{i+k}$ so for every vertex and edge the embedding eventually stabilizes. The resulting embedding has 1 accumulation point, the center of $D_i$.
\end{proof}

\begin{proof}[Proof of Theorem \ref{starminor}]
This is a special case of the first claim in Theorem \ref{theorem:simply_connected_comb_embedding}.
\end{proof}

\section{Unimodular combinatorial embeddings of URPG's}\label{section:unimod_combin}

From now on we only consider embeddings into $S^2$.

In this section we prove Theorem \ref{theorem:unimodular_combinatorial_embedding_1_acc}. The titles of the subsections are meant to provide a sketch of the proof. The first three subsections give the necessary definitions and earlier results that we will be relying on. 

\subsection{Accumulation points of planar combinatorial embeddings}
Let $G$ be an infinite, locally finite graph. Given a combinatorial embedding of $G$, can we identify accumulation points that will be present for any actual embedding representing this combinatorial embedding?

In Section \ref{subsection:combinatorial_embeddings} we argued that for general surfaces a combinatorial embedding is not enough to reconstruct an embedding even up to homeomorphisms. In the case of $S^2$, however, the reconstruction is possible when the graph is finite, because all combinatorial faces bound a topological face homeomorphic to a disk. 

Unfortunately this does not hold for infinite graphs. Indeed, there are embeddings of $\Z$ in $S^2$ with 1 or with 2 accumulation points, while both embeddings define the same combinatorial embedding. This means that for a combinatorial embedding $(\pi_v)_{v \in V(G)}$ the number of accumulation points of an embedding $\iota$ realizing $\pi$ is not well defined. We are, however, interested in embeddings with as few accumulation points as possible.

We proceed to define $\acc(G, \pi)$, the number of accumulation points of a combinatorial embedding, using only the combinatorial data. It will turn out that $\acc(G, \pi)$ is exactly the minimal number of accumulation points needed to realize $\pi$.

Given a cycle on darts we define the inside of that cycle. Let $C$ be a cyclically ordered tuple of distinct darts $(\righte_1, \ldots \righte_n)$ such that $t(\righte_i) = s(\righte_{i+1})=v_{i+1}$ for $1 \leq i \leq n$. Let $\righte \notin C$ be a dart with $s(\righte)= v_i$. 
We say that $\righte$ is \emph{inside} $C$ if $\righte$ is somewhere between $\lefte_{i-1}$ and $\righte_i$ in the cyclic order $\pi_{v_i}$. If $t(\righte) \notin \{v_1, \ldots, v_n\}$ (i.e.\ $\righte$ is not a chord of $C$), then the whole connected component of $t(\righte)$ in the graph $G \setminus C$ is said to be inside $C$. This definition does not depend on the choice of $\righte$ connecting the component to $C$.

The outside of a cycle is defined analogously. Note that inverting all darts along the cycle also inverts the notion of inside and outside. When the graph is finite, a combinatorial face is a cycle that has nothing inside it. We will say a part of a graph is \emph{surrounded} by a cycle if it is inside that cycle. 

Consider a finite connected subgraph $H \subset G$. It is endowed with the combinatorial embedding inherited from $G$. The cyclical order $\pi^H_{v}$ at a vertex $v \in V(H)$ on the adjacent $H$-darts is the restriction of the cyclical order $\pi_v$ on the adjacent $G$-darts. Hence one can also define the faces of $H$, inherited from the combinatorial embedding of $G$. These may not be faces of $G$, of course.

For a connected component of $G \setminus H$ there is a unique face of $H$ that surrounds it. Let $a^{\pi}_H$ denote the number of faces of $H$ that surround at least one infinite component of $G \setminus H$. It is  straightforward to check that if $H \subseteq H'$, where $H'$ is also finite, then $a^{\pi}_H \leq a^{\pi}_{H'}$. This leads us to the following.

\begin{definition}\label{def:comb_acc}
Let $H_1 \subset H_2 \subset \ldots \subset G$ be an exhaustion of $G$ by finite graphs. We define $\acc(G,\pi)=\sup a^{\pi}_{H_i}$. The value of $\acc(G,\pi)$ does not depend on the particular exhaustion because of the monotonicity. When the combinatorial embedding is clear from context, we write $\acc(G)$.
\end{definition}

Any embedding of $G$ into $S^2$ that realizes $\pi$ has at least $\acc(G,\pi)$ accumulation points. On the other hand when $\acc(G,\pi)$ is finite, there exists an embedding realizing $\pi$ with exactly $\acc(G,\pi)$ accumulation points. Indeed, suppose that $a^{\pi}_{H_i}$ stabilizes at $i_0$. We can make sure that accumulation points only occur at the centers of the $\acc(G,\pi)$-many topological faces of $H_{i_0}$ that surround infinitely many points of $G \setminus H_{i_0}$. The reasoning is the same as the last part of the proof of Theorem \ref{theorem:simply_connected_comb_embedding}.

\subsection{The general strategy}

In \cite{timar2023unimodular} the first author shows that URPG's have unimodular combinatorial embeddings into the plane. Notice that when the URPG has one end, \emph{any} combinatorial embedding into $S^2$ gives $\acc(G)=1$. 

Here, in Theorem \ref{theorem:unimodular_combinatorial_embedding_1_acc}, we show that as long as the URPG $(G,o)$ is supported on $\G_{S^2}^*$, it has a unimodular random combinatorial embedding into the plane with $\acc(G) \leq 1$ almost surely.

Recall that an embedding with one topological accumulation point is called \emph{simply connected}. For example the $d$-regular tree $T_d$ ($d \geq 2$) has continuum many ends, but admits simply connected embeddings into $S^2$. On the other hand $K_2 \times T_3$ does not. 

Throughout this Section we consider graphs of increasing generality and build up to a proof of Theorem \ref{theorem:unimodular_combinatorial_embedding_1_acc}. As in \cite{timar2023unimodular}, we construct the random embedding by decomposing our graph into 2-connected, and then further into 3-connected components. We choose random combinatorial embeddings for those parts and glue them together to form a combinatorial embedding of the whole graph. 

The main reason why our construction works is that multi-ended infinite $2$-connected components essentially have a unique combinatorial embedding with $1$ accumulation point whenever they have at least one. (See Lemma \ref{lemma:full_core}.) Technical difficulties arise because we have to keep track of where ends get embedded.

\subsection{Building blocks of graphs: Tutte and block-cut decompositions}

The following are the main tools for constructing a random combinatorial embedding in \cite[Section 2]{timar2023unimodular}. References and a more detailed introduction can be found there, here we try to tighten the exposition.

\subsubsection{3-connected components}
The starting point of the construction is the fact (due to Whitney and Imrich) that a $3$-connected graph has a unique combinatorial embedding, up to inverting all permutations, even in the infinite case.

\begin{example}
To find examples of 3-connected graphs with infinitely many ends that have a locally finite embedding do the following. Start with any (possibly unimodular random) graph with infinitely many ends and a locally finite planar embedding, and assume for simplicity that it has no leaves (e.g.\ regular-, or unimodular Galton-Watson trees). Perform the local operation shown in Figure \ref{fig:3-connected-examples} around every vertex to turn it into a 3-connected graph while preserving the other properties.  
\end{example}

\begin{figure}[ht]
    \centering
	\begin{tikzpicture}[node distance=1cm,
	main node/.style={circle,minimum width=1.9pt, fill, 		inner sep=0pt,outer sep = 0pt},
	red node/.style={circle,minimum width=2.2pt, fill,red, 	inner sep=0pt,outer sep = 0pt},
	blue node/.style={circle,minimum width=2.2pt, 				fill,blue, inner sep=0pt,outer sep = 0pt},
	square node/.style={draw,regular polygon,regular 			polygon sides=4,fill,inner sep=0.9pt,outer sep=0.1pt},
	triangle node/.style={draw,regular polygon,regular 			polygon sides=3,fill,inner sep=0.8pt,outer sep=0.1pt}]   

  \node[main node] (o) at (0,0){};

  \path [draw=gray] (o) edge (-0.66,0.77);
  \path [draw=gray] (o) edge (0.25,0.97);
  \path [draw=gray] (o) edge (0.93,-0.37);
  \path [draw=gray] (o) edge (-0.13,-0.99);
  \path [draw=gray] (o) edge (-0.98,-0.2);
  
  \path [draw=gray, thin, dotted] (-0.99,1.155) edge (-0.66,0.77);
  \path [draw=gray, thin, dotted] (0.375,1.455) edge (0.25,0.97);
  \path [draw=gray, thin, dotted] (1.395,-0.555) edge (0.93,-0.37);
  \path [draw=gray, thin, dotted] (-0.195,-1.485) edge (-0.13,-0.99);
  \path [draw=gray, thin, dotted] (-1.47,-0.3) edge (-0.98,-0.2);
  
  \node[main node] (o1) at (-0.115,0.4865){};
  \node[main node] (o2) at (0.445,0.226){};
  \node[main node] (o3) at (0.253,-0.43){};
  \node[main node] (o4) at (-0.341,-0.365){};
  \node[main node] (o5) at (-0.472,0.164){};
  
  \path [draw=gray, thin] (o1) edge (-0.775,1.2565);
  \path [draw=gray, thin] (o2) edge (0.695,1.196);
  \path [draw=gray, thin] (o3) edge (1.183,-0.8);
  \path [draw=gray, thin] (o4) edge (-0.471,-1.355);
  \path [draw=gray, thin] (o5) edge (-1.452,-0.036);

  \path [draw=gray, thin] (o5) edge (-1.132,0.934);
  \path [draw=gray, thin] (o1) edge (0.135,1.4565);
  \path [draw=gray, thin] (o2) edge (1.375,-0.144);
  \path [draw=gray, thin] (o3) edge (0.123,-1.42);
  \path [draw=gray, thin] (o4) edge (-1.321,-0.565);

  \path [draw=gray, thin] (o) edge (o1);
  \path [draw=gray, thin] (o) edge (o2);
  \path [draw=gray, thin] (o) edge (o3);
  \path [draw=gray, thin] (o) edge (o4);
  \path [draw=gray, thin] (o) edge (o5);

  \draw [thin, gray, -stealth] (-3.5,0) -- (-2.5,0);

\node[main node] (p) at (-6,0){};

  \path [draw=gray] (p) edge (-6.66,0.77);
  \path [draw=gray] (p) edge (-5.75,0.97);
  \path [draw=gray] (p) edge (-5.07,-0.37);
  \path [draw=gray] (p) edge (-6.13,-0.99);
  \path [draw=gray] (p) edge (-6.98,-0.2);
  
  \path [draw=gray, thin, dotted] (-6.99,1.155) edge (-6.66,0.77);
  \path [draw=gray, thin, dotted] (-5.625,1.455) edge (-5.75,0.97);
  \path [draw=gray, thin, dotted] (-4.605,-0.555) edge (-5.07,-0.37);
  \path [draw=gray, thin, dotted] (-6.195,-1.485) edge (-6.13,-0.99);
  \path [draw=gray, thin, dotted] (-7.47,-0.3) edge (-6.98,-0.2);

	\end{tikzpicture}
    \caption{Local operation ensuring 3-connectedness}
    \label{fig:3-connected-examples}
\end{figure}

\subsubsection{Tutte-decomposition}
Any $2$-connected graph $G$ has a unique Tutte-decomposition into $3$-connected components, 3-links and cycles. A 3-link is a graph consisting of 3 parallel edges.

By the \emph{amalgam} of two graphs along an edge we mean taking their disjoint union, picking an edge in both (together with a bijection between their endpoints), deleting those edges, and identifying the endpoints. Amalgamation of graphs $A$ and $B$ is denoted by $A+B$, the edge will be clear from context.

Informally speaking, the Tutte decomposition describes a treelike way of obtaining $G$ by amalgamating $3$-connected components and cycles, possibly keeping (1 copy of) the edge we are amalgamating along. We can achieve this by inserting a $3$-link between two graphs during the amalgamating procedure whenever we want to keep the common edge. See Figure \ref{fig:tutte_decomp_ladder} for the Tutte decomposition of the infinite ladder.

\begin{figure}[h]
     \centering
         \begin{tikzpicture}[node distance=1cm,
main node/.style={circle,minimum width=1.9pt, fill, inner sep=0pt,outer sep = 0pt},
red node/.style={circle,minimum width=2.2pt, fill,red, inner sep=0pt,outer sep = 0pt},
blue node/.style={circle,minimum width=2.2pt, fill,blue, inner sep=0pt,outer sep = 0pt},
square node/.style={draw,regular polygon,regular polygon sides=4,fill,inner sep=0.9pt,outer sep=0.1pt},
triangle node/.style={draw,regular polygon,regular polygon sides=3,fill,inner sep=0.8pt,outer sep=0.1pt}]   

  \node (d) at (-3,0)[label=center:{$\cdots$}]{};
  \node (e) at (3,0)[label=center:{$\cdots$}]{};
  \node[main node] (a) at (-1.4,0.7){};
  \node[main node] (b) at (0,0.7){};
  \node[main node] (c) at (1.4,0.7){};
  \node[main node] (u) at (-1.4,-0.7){};
  \node[main node] (v) at (0,-0.7){};
  \node[main node] (w) at (1.4,-0.7){};

  \path [draw=gray, thin] (a) edge (b);
  \path [draw=gray, thin] (b) edge (c);
  \path [draw=gray, thin] (a) edge (u);
  \path [draw=gray, thin] (b) edge (v);
  \path [draw=gray, thin] (c) edge (w);
  \path [draw=gray, thin] (u) edge (v);
  \path [draw=gray, thin] (v) edge (w);
  \path [draw=gray, thin] (-2.8,0.7) edge (a);
  \path [draw=gray, thin] (-2.8,-0.7) edge (u);
  \path [draw=gray, thin] (c) edge (2.8,0.7);
  \path [draw=gray, thin] (w) edge (2.8,-0.7);

  \draw [thin, gray, -stealth] (0,-1.1) -- (0,-1.9);

  \node[main node] (A1) at (-4.4,-2.3){};
  \node[main node] (A2) at (-5.8,-2.3){};
  \node[main node] (A3) at (-4.4,-3.7){};
  \node[main node] (A4) at (-5.8,-3.7){};
  \node (A) at (-5.1, -4.2) [label=center:{$G_{\alpha}$}]{};

  \node[main node] (X1) at (-3.4,-2.3){};
  \node[main node] (X2) at (-3.4,-3.7){};
  \node (X) at (-3.4, -4.2) [label=center:{$G_{\beta}$}]{};

  \node[main node] (B1) at (-1,-2.3){};
  \node[main node] (B2) at (-2.4,-2.3){};
  \node[main node] (B3) at (-1,-3.7){};
  \node[main node] (B4) at (-2.4,-3.7){};
  
  \node[main node] (Y1) at (0,-2.3){};
  \node[main node] (Y2) at (0,-3.7){};

  \node[main node] (C1) at (1,-2.3){};
  \node[main node] (C2) at (2.4,-2.3){};
  \node[main node] (C3) at (1,-3.7){};
  \node[main node] (C4) at (2.4,-3.7){};
  
  \node[main node] (Z1) at (3.4,-2.3){};
  \node[main node] (Z2) at (3.4,-3.7){};

  \node[main node] (D1) at (4.4,-2.3){};
  \node[main node] (D2) at (5.8,-2.3){};
  \node[main node] (D3) at (4.4,-3.7){};
  \node[main node] (D4) at (5.8,-3.7){};

  \path [draw=gray, thin] (A1) edge (A2);
  \path [draw=gray, thin] (A2) edge (A4);
  \path [draw=gray, thin] (A4) edge (A3);
  \path [draw=blue, thin] (A3) edge (A1);

  \path [draw=gray, thin] (B1) edge (B2);
  \path [draw=red, thin] (B2) edge (B4);
  \path [draw=gray, thin] (B4) edge (B3);
  \path [draw=green, thin] (B3) edge (B1);

  \path [draw=gray, thin] (C1) edge (C2);
  \path [draw=purple, thin] (C2) edge (C4);
  \path [draw=gray, thin] (C4) edge (C3);
  \path [draw=orange, thin] (C3) edge (C1);

  \path [draw=gray, thin] (D1) edge (D2);
  \path [draw=gray, thin] (D2) edge (D4);
  \path [draw=gray, thin] (D4) edge (D3);
  \path [draw=brown, thin] (D3) edge (D1);
  
  \path [draw=blue, thin] (X1) to [bend right] (X2);
  \path [draw=gray, thin] (X1) edge (X2);
  \path [draw=red, thin] (X1) to [bend left] (X2);

  \path [draw=green, thin] (Y1) to [bend right] (Y2);
  \path [draw=gray, thin] (Y1) edge (Y2);
  \path [draw=orange, thin] (Y1) to [bend left] (Y2);

  \path [draw=purple, thin] (Z1) to [bend right] (Z2);
  \path [draw=gray, thin] (Z1) edge (Z2);
  \path [draw=brown, thin] (Z1) to [bend left] (Z2);

  \draw [thin, gray, -stealth] (-4.2,-2.5) -- (-3.7,-2.5);
  \node (AX) at (-4,-2) [label=center:{$f(\alpha,\beta)$}]{};

  \node (D) at (-6.8,-3)[label=center:{$\cdots$}]{};
  \node (E) at (6.8,-3)[label=center:{$\cdots$}]{};

\end{tikzpicture}
        \caption{The Tutte decomposition of the infinite ladder}
        \label{fig:tutte_decomp_ladder}
\end{figure}

More precisely, the Tutte decomposition gives a tree $T$ with vertices $\alpha \in V_T$ labelled by graphs $G_\alpha$, and edges $(\alpha,\beta)$ labelled by functions $f(\alpha,\beta)$. The $G_{\alpha}$ are 3-connected, 3-links, or cycles, and the function $f(\alpha,\beta)$ picks an edge from $G_{\alpha}$ and $G_{\beta}$, and a bijection between their endpoints. Moreover, each edge of each $G_{\alpha}$ is picked by at most one $f(\alpha,\beta)$. The order of amalgamations is arbitrary, and we denote the graph obtained in the end by $\Gamma(T)$. We say an edge in some $G_{\alpha}$ is \emph{virtual}, if it is selected by some $f(\alpha,\beta)$. An edge of some $G_{\alpha}$ is present in $\Gamma(T)$ if and only if it is not virtual. As amalgams of cycles are cycles, we also assume that cycles cannot be neighbors in the Tutte tree. This guarantees uniqueness of the decomposition.

\subsubsection{Block-cut decompositions}

Finally, if $G$ is a locally finite, connected planar graph, it has a block-cut tree decomposition into 2-connected components. The 2-connected components (blocks) $G_a$ are indexed by some set $A$. The cut vertices form the set $C$. The vertex set of the block-cut tree $\mathcal{T}$ is  $A \cup C$, and $a \in A$ is connected to $v \in C$ by an edge if $v \in V(G_a)$. In $\mathcal{T}$ the vertices $a,b \in A$ are at distance 2 if and only if $G_{a}$ and $G_{b}$ have a (unique) common vertex.

\subsection{2-connected graphs with full core}

In this subsection we define the \emph{core} of 2-connected infinite graphs, and investigate combinatorial embeddings of such graphs with {\it full core}. 

\begin{definition*}
Let $G$ be an infinite, 2-connected graph and $T=(V_T, E_T)$ the Tutte tree with $\Gamma(T)=G$. Let $V_{\infty} \subseteq V_T$ denote the set of vertices that represent infinite 3-connected Tutte components. Let $\core(T)$ denote the convex hull of $V_{\infty} \cup \Ends(T)$. That is, a vertex $\alpha \in V_T$ is in $\core(T)$ if and only if $\alpha$ is either contained in a path between two vertices in $V_{\infty}$ (allowing also the case $\alpha \in V_{\infty}$), or contained in an infinite path starting at a vertex in $V_{\infty}$, or contained in a bi-infinite path. Let $\core(G)=\Gamma(\core(T))$.
\end{definition*}

\begin{remark}
Note that $\core(T)$ can be obtained by deleting all finite parts of $T$ that contain no vertices of $V_{\infty}$ and can be separated from $T$ by deleting an edge.
\end{remark}

\begin{example}
We have seen the Tutte decomposition of the infinite ladder in Figure \ref{fig:tutte_decomp_ladder}. In this case the Tutte tree $T$ is a bi-infinite path, so $T=\core (T)$, and consequently $G=\core(G)$. 
\end{example}

\begin{example}
Another 2-connected example with infinitely many ends can be seen in Figure \ref{fig:2-connected_example}. The core is the 4-regular-tree-like part drawn with black, while the rest of the graph, consisting of finite components amalgamated at some places, is drawn with red.
\end{example}

\begin{figure}[h]
     \centering
         \begin{tikzpicture}[node distance=1cm,
main node/.style={circle,minimum width=1.9pt, fill, inner sep=0pt,outer sep = 0pt},
red node/.style={circle,minimum width=2.2pt, fill,red, inner sep=0pt,outer sep = 0pt},
blue node/.style={circle,minimum width=2.2pt, fill,blue, inner sep=0pt,outer sep = 0pt},
square node/.style={draw,regular polygon,regular polygon sides=4,fill,inner sep=0.9pt,outer sep=0.1pt},
triangle node/.style={draw,regular polygon,regular polygon sides=3,fill,inner sep=0.8pt,outer sep=0.1pt}]   

  \node[main node] (A1) at (0,0){};
  \node[main node] (A2) at (0,1){};
  \node[main node] (A3) at (1,0){};
  \node[main node] (A4) at (1,1){};

  \node[main node] (B1) at (2.25,0.25){};
  \node[main node] (B2) at (2.25,0.75){};
  \node[main node] (B3) at (2.75,0.25){};
  \node[main node] (B4) at (2.75,0.75){};
  
  \node[main node] (C1) at (0.25,2.25){};
  \node[main node] (C2) at (0.25,2.75){};
  \node[main node] (C3) at (0.75,2.25){};
  \node[main node] (C4) at (0.75,2.75){};
  
  \node[main node] (D1) at (-1.75,0.25){};
  \node[main node] (D2) at (-1.75,0.75){};
  \node[main node] (D3) at (-1.25,0.25){};
  \node[main node] (D4) at (-1.25,0.75){};  

  \node[main node] (E1) at (0.25,-1.75){};
  \node[main node] (E2) at (0.25,-1.25){};
  \node[main node] (E3) at (0.75,-1.75){};
  \node[main node] (E4) at (0.75,-1.25){};

  \path [draw=gray, thin] (A1) edge (A2);
  \path [draw=gray, thin] (A1) edge (A3);
  \path [draw=gray, thin] (A4) edge (A2);  
  \path [draw=gray, thin] (A4) edge (A3);
  
  \path [draw=gray, thin] (B1) edge (B2);
  \path [draw=gray, thin] (B1) edge (B3);
  \path [draw=gray, thin] (B4) edge (B2);  
  \path [draw=gray, thin] (B4) edge (B3);
  
  \path [draw=gray, thin] (C1) edge (C2);
  \path [draw=gray, thin] (C1) edge (C3);
  \path [draw=gray, thin] (C4) edge (C2);  
  \path [draw=gray, thin] (C4) edge (C3);

  \path [draw=gray, thin] (D1) edge (D2);
  \path [draw=gray, thin] (D1) edge (D3);
  \path [draw=gray, thin] (D4) edge (D2);  
  \path [draw=gray, thin] (D4) edge (D3);

  \path [draw=gray, thin] (E1) edge (E2);
  \path [draw=gray, thin] (E1) edge (E3);
  \path [draw=gray, thin] (E4) edge (E2);  
  \path [draw=gray, thin] (E4) edge (E3);

  \path [draw=gray, thin] (A1) edge (E2);
  \path [draw=gray, thin] (A3) edge (E4);

  \path [draw=gray, thin] (A1) edge (D3);
  \path [draw=gray, thin] (A2) edge (D4);

  \path [draw=gray, thin] (A2) edge (C1);
  \path [draw=gray, thin] (A4) edge (C3);

  \path [draw=gray, thin] (A4) edge (B2);
  \path [draw=gray, thin] (A3) edge (B1);
  
  \node[red node] (X1) at (0.6,0.7){};
  \node[red node] (X2) at (0.6,0.3){};  
  \node[red node] (Y) at (-0.5,1.1){};
  \node[red node] (Z) at (0.5,-0.5){};

  \path [draw=red, thin] (A4) edge (X1);
  \path [draw=red, thin] (X2) edge (X1);  
  \path [draw=red, thin] (A3) edge (X2);  
  \path [draw=red, thin] (A2) edge (Y);
  \path [draw=red, thin] (D4) edge (Y);
  \path [draw=red, thin] (A1) edge (Z);
  \path [draw=red, thin] (E2) edge (Z);    

  \node (E) at (-0.25,-1.5)[label=center:{$\cdots$}]{};
  \node (E') at (0.5,-2.2)[label=center:{$\vdots$}]{};
  \node (E'') at (1.3,-1.5)[label=center:{$\cdots$}]{};

  \node (C) at (-0.25, 2.5)[label=center:{$\cdots$}]{};
  \node (C') at (0.5, 3.4)[label=center:{$\vdots$}]{};
  \node (C'') at (1.3, 2.5)[label=center:{$\cdots$}]{};

  \node (B) at (3.35, 0.5)[label=center:{$\cdots$}]{};
  \node (B') at (2.5,1.4)[label=center:{$\vdots$}]{};
  \node (B'') at (2.5,-0.2)[label=center:{$\vdots$}]{};
  
  \node (D) at (-2.25, 0.5)[label=center:{$\cdots$}]{};
  \node (D') at (-1.5,1.4)[label=center:{$\vdots$}]{};
  \node (D'') at (-1.5,-0.2)[label=center:{$\vdots$}]{};  

\end{tikzpicture}
        \caption{A $2$-connected graph with infinitely many ends and a simply connected embedding}
        \label{fig:2-connected_example}
\end{figure}

\begin{lemma} \label{lemma:full_core}
Let $G$ be an infinite, $2$-connected planar graph such that $\core(G) = G$. Assume also that $G$ has a simply connected embedding into $S^2$. Then (up to inverting all permutations) it has a unique combinatorial embedding with $\acc(G)=1$.
\end{lemma}

\begin{proof}
Let $T$ denote the Tutte tree of $G$, so $G= \Gamma(T)$. Since $G$ is planar, all the Tutte components are planar. Also any amalgam of these obtained by performing some amalgamations as prescribed by $T$ is planar. 

Most importantly, as it is proved in \cite[Lemma 12]{timar2023unimodular}, a combinatorial embedding $\pi^{G}$ of $G$ is uniquely determined by its restrictions to the Tutte components. We denote these restrictions $(\pi^{\alpha})_{\alpha \in V(T)}$. For $3$-connected components and $3$-links we have $2$ possible combinatorial embeddings, while for cycles we have $1$.

We know that $G$ has a simply connected embedding. Pick one and let $\pi^{G}_{+}$ denote the corresponding combinatorial embedding. Accordingly, the restrictions are denoted $\pi^{\alpha}_{+}$. When $G_{\alpha}$ is a $3$-connected component or a $3$-link, let $\pi^{\alpha}_{-}$ denote the other combinatorial embedding of $G_{\alpha}$. If $G_{\alpha}$ is a cycle, we put $\pi^{\alpha}_{-}=\pi^{\alpha}_{+}$. Let $\pi^{G}_{-}$ denote the combinatorial embedding of $G$ with restrictions $\pi^{\alpha}_{-}$ for all $\alpha \in V(T)$. Since $\pi^{G}_{-}$ is obtained from $\pi^{G}_{+}$ by inverting all permutations, it also belongs to a simply connected embedding of $G$.

As we have said, to specify a combinatorial embedding $\pi^G$ of $G$ one has to make a choice between $\pi^{\alpha}_{+}$ and $\pi^{\alpha}_{-}$ for all $\alpha \in V(T)$. (For cycles this means no choice.) We claim that from these possible collections of choices
only two give $\acc(G,\pi^G)=1$, namely $\pi^{G}_{+}$ where we choose $\pi^{\alpha}_{+}$ for all $\alpha$, and $\pi^{G}_{-}$, where we choose $\pi^{\alpha}_{-}$ for all $\alpha$. 

Suppose $\pi^{G}$ is neither $\pi^{G}_{+}$ nor $\pi^{G}_{-}$. We aim to find a cycle in $(G, \pi^G)$ which has infinitely many points both on the inside and outside. Such a cycle shows $ \acc(G,\pi^G)\geq 2$.
Now, we can find $G_{\alpha}$ and $G_{\beta}$, $3$-connected components or $3$-links in the Tutte decomposition such that $\pi^{\alpha} = \pi^{\alpha}_{+}$ and $\pi^{\beta}=\pi^{\beta}_{-}$. Moreover, we can pick $\alpha$ and $\beta$ to be neighbors in $T$, or at worst at distance two in $T$ with some $\gamma$ in between, with $G_{\gamma}$ a cycle. Assume first that $\alpha$ and $\beta$ are neighbors. The proof will not be essentially different when there is $\gamma$ in between. 

Consider the case when both $G_{\alpha}$ and $G_{\beta}$ are finite. 
Let $e$ denote the virtual edge of $G_{\alpha}$ that we use to amalgamate it to $G_{\beta}$. As $\alpha \in \core (T)$, we have another virtual edge of $G_{\alpha}$, say $f$. Let $T_e$ and $T_f$ denote the components of $T \setminus \alpha$ corresponding to $e$ and $f$. We know that $\Gamma(T_e)$ and $\Gamma(T_f)$ are both infinite by our assumption that $G=\core(G)$ and consequently $T= \core(T)$. Let $E=(\righte, \righte_2, \ldots , \righte_n)$ and $F = (\lefte, \rightf_2, \ldots, \rightf_k)$ denote the two faces of $G_{\alpha}$ adjacent to $e$. Set $v_1=s(\righte)$ and $v_2=t(\righte)$.

The only edge that $E$ and $F$ share is $e$, hence
\[C:=E + F=(\righte_2, \ldots, \righte_n, \rightf_2, \ldots, \rightf_k)\] is a cycle in $G_{\alpha}$. The cycle $C$ surrounds $e$, therefore when $G_{\beta}$ is amalgamated to $G_{\alpha}$ along $e$, the whole of $G_{\beta} \setminus e$ will be inside $C$.

\begin{figure}[h] \label{figure:alpha}
	\centering 
		\includegraphics[height=6cm]{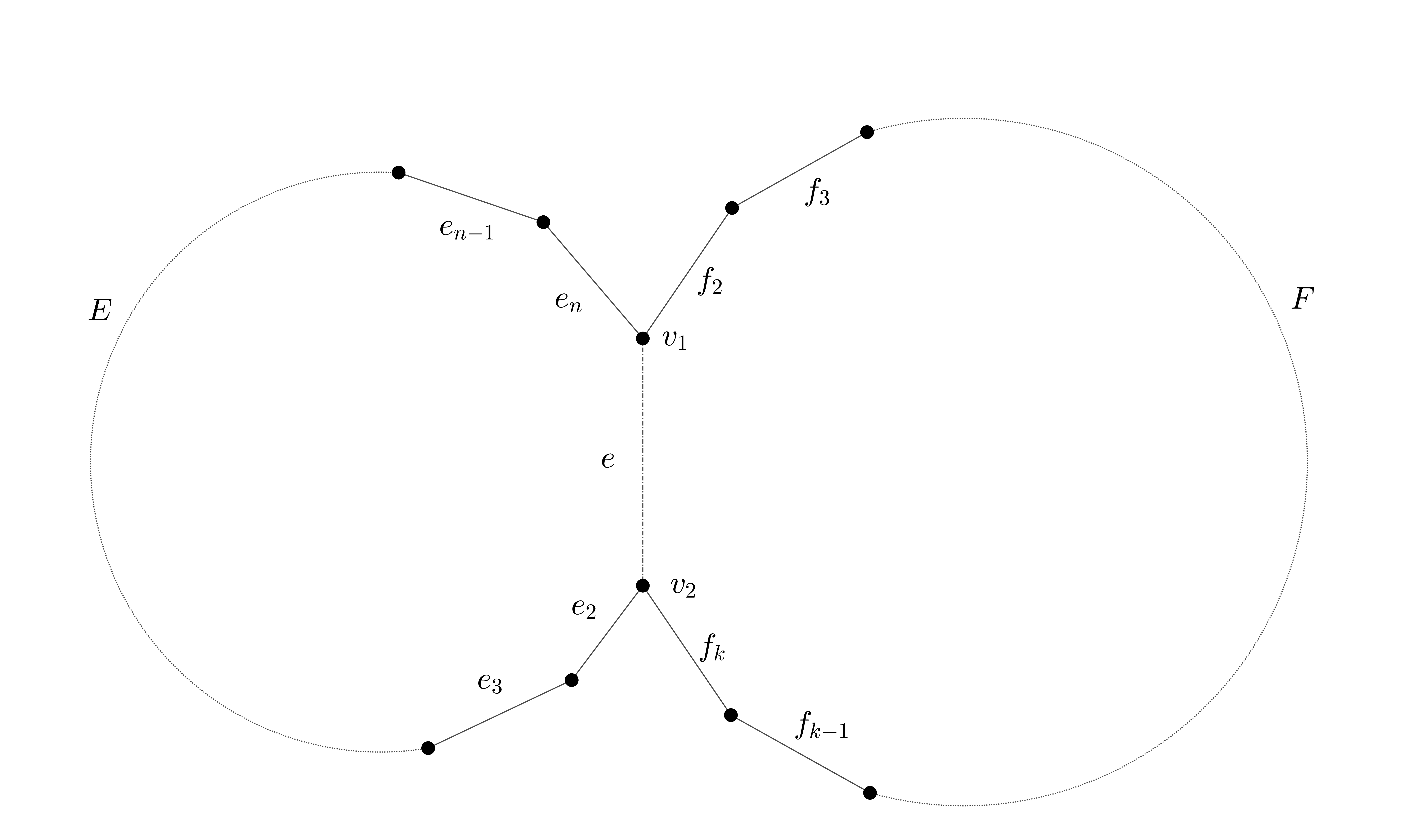}		
		\caption{The faces $E$ and $F$ in $G_{\alpha}, \pi^{\alpha}_+$}
\end{figure}

Similar to $G_{\alpha}$, the Tutte component $G_{\beta}$ also has at least one other virtual edge $f'$ apart from $e$. Clearly $f \neq f'$, as the only edge shared by $G_{\alpha}$ and $G_{\beta}$ is $e$. What happens to the faces $E$ and $F$ when we amalgamate along $e$? Because $G_{\beta}$ is either $3$-connected or a $3$-link we can find three edge-disjoint paths in $G_{\beta}$ between the endpoints of $e$. At least one of these avoids both $e$ and $f'$, let us denote it by $P=(\rightp_1,\ldots, \rightp_l)$.  With slight abuse of notation we use $v_1=s(\rightp_1)$ and $v_2=t(\rightp_l)$ to denote the endpoints of $e$ that are both in $G_\beta$ and $G_\alpha$. Note that $P$ might contain other virtual edges of $\beta$, different from $e$ and $f'$, but that will not cause problems for us. 

In $G_{\alpha} + G_{\beta}$ the edge $e$ is deleted, so the $G_{\alpha}$-faces $E$ and $F$ merge into the cycle $C = E + F$, but the path $P$ splits it into two cycles $E'$ and $F'$ in $G_{\alpha} + G_{\beta}$. (Recall that $P \subseteq G_\beta$ is inside $C$.) To be precise let $E'=(\rightp_1, \ldots, \rightp_l, \righte_2, \ldots, \righte_n)$ and $F'=(\leftp_l, \ldots, \leftp_1, \rightf_2, \ldots, \rightf_k)$. 
 
We know that $f'$ is inside $C$ and not a part of $P$, so it is inside either $E'$ or $F'$, depending on the choice of combinatorial embedding $\pi^{\beta}$. Assume that when we amalgamate according to $\pi^{\beta}_{+}$ the $\beta$-edge $f'$ falls inside the $E'$ cycle, and when we amalgamate according to $\pi^{\beta}_{-}$ it falls inside $F'$.

\begin{figure}[h] 
		\centering 
		\includegraphics[height=6cm]{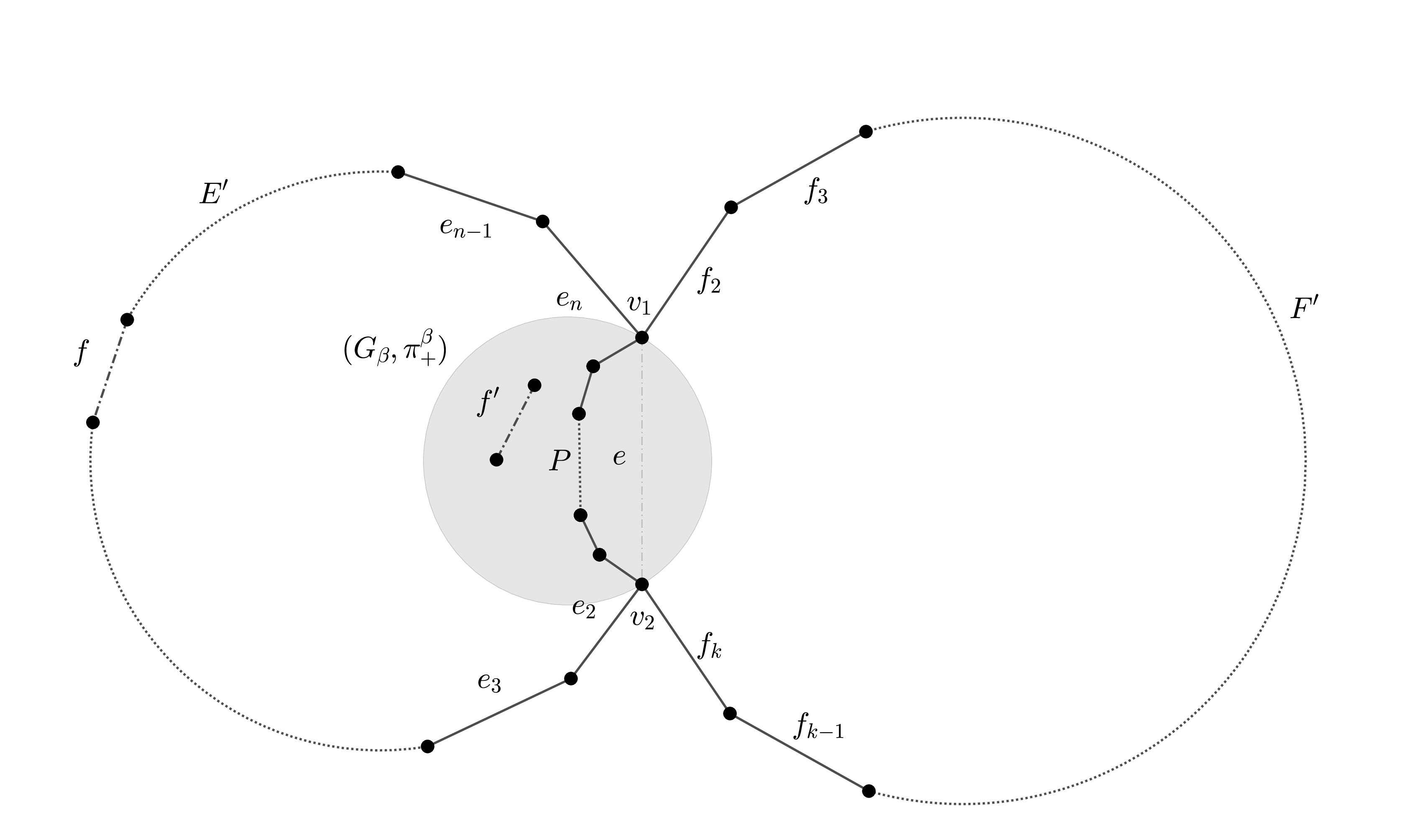}
		\caption{The amalgam $(G_{\alpha}+G_{\beta}, \pi^{\alpha}_+, \pi^{\beta}_+)$}
		\label{figure:alphabeta}
\end{figure}

When we amalgamate according to $\pi^{G}_+$, in $G_{\alpha} + G_{\beta}$ the virtual edges $f$ and $f'$ must not be separated by a cycle, because that would indicate $\acc(G,\pi_{+}^G) \geq 2$. This means that $f$ also has to be along the face $E$ in $G_{\alpha}$, otherwise $E'$ would separate them in $G_{\alpha}+G_{\beta}$. See Figure \ref{figure:alphabeta}. 

We claim that $f$ cannot be part of the face $F$ in $G_{\alpha}$. Indeed, in a $3$-connected planar graph one cannot find two edges that have two common faces. This also holds for the $3$-link. If $f$ was incident to $F$, then $e$ and $f$ would be such edges.

We use the fact that $f$ is not incident to $F$ in $G_{\alpha}$ to establish $\acc(G,\pi^G)\geq 2$ when we amalgamate according to $\pi^{\alpha}_{+}$ and $\pi^{\beta}_{-}$. In this case we have $f'$ surrounded by the cyclic walk $F'$, see Figure \ref{figure:alphabeta_neg}.
Therefore it is separated from $f$ in $G_{\alpha} + G_{\beta}$, since $f$ is along $E'$, outside $F'$. Irrespective of the other choices of $\pi^{G}$ the two separated virtual edges already indicate $\acc(G,\pi^G)\geq 2$. This finishes the proof in the finite case.

\begin{figure}[h] 
		\centering 
		\includegraphics[height=6cm]{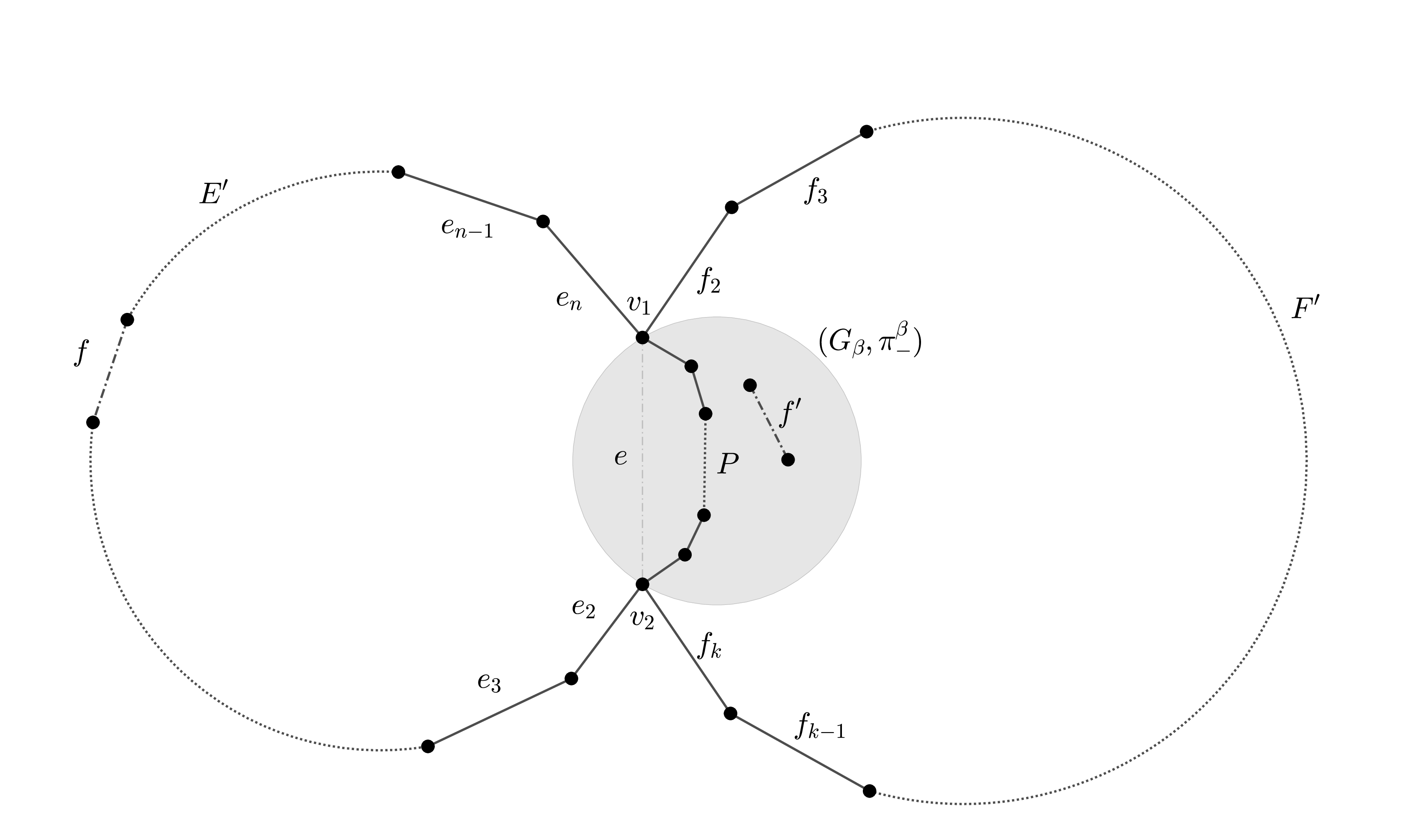}
		\caption{The amalgam $(G_{\alpha}+G_{\beta}, \pi^{\alpha}_+, \pi^{\beta}_{-})$}
		\label{figure:alphabeta_neg}
\end{figure}

If at least one of $G_{\alpha}$ and $G_{\beta}$ is infinite, we can assume without loss of generality that $G_{\beta}$ is infinite. In $G_{\alpha}$, using the $3$-connectedness (without assuming finiteness) we find two disjoint directed paths $\overrightarrow{P}_1$ and $\overrightarrow{P}_2$ evading $e$ from $v_2$ to $v_1$, and form the two cycles $E=(\righte, \overrightarrow{P}_1)$ and $F=(\lefte, \overleftarrow{P}_2)$. We proceed the same way as in the finite case. We either find a virtual edge $f$ of $G_{\alpha}$, or $G_{\alpha}$ is infinite as well. In this case it has infinitely many points, almost all of which have to fall inside $C = (\overrightarrow{P}_1, \overleftarrow{P}_2)$. The argument leads to $F'$ separating infinitely many points of $G_{\beta}$ from infinitely many points of $G_{\alpha}$ (or $f$).

Finally if $\alpha$ and $\beta$ are at distance two with $\gamma$ in between,  we proceed almost identically. We consider $G_{\alpha}$ and $G_{\beta} + G_{\gamma}$, and since $G_{\gamma}$ is a cycle, $G_{\beta} + G_{\gamma}$ is a subdivision of $G_{\beta}$. Let $e$ be the common virtual edge of $G_{\alpha}$ and $G_{\gamma}$, and $e'$ the common virtual edge of $G_{\beta}$ and $G_{\gamma}$. When choosing the path $P$ in $G_{\beta} + G_{\gamma}$ to replace $e$ in the $G_{\alpha}$-cycles $E$ and $F$ we first find the path $P_0$ in $G_{\beta}$ evading $e'$ and $f'$, and then complement it with $G_{\gamma} \setminus \{e,e'\}$ to get a path connecting the endpoints of $e$.
\end{proof}

\subsection{General 2-connected planar graphs} \label{subsection:general_2-connected}

Let $G$ be a 2-connected planar graph with at least $2$ ends that has a simply connected embedding. Let $T$ denote its Tutte tree. Then $\core(T)$ is nonempty, and $\core(G)$ also has a simply connected embedding. By Lemma \ref{lemma:full_core} there are exactly two combinatorial embeddings with $\acc(\core(G))=1$, using the fact that $\core (G)$ is always connected. As before, we denote the induced combinatorial embeddings on the Tutte components by $\pi^{\alpha}_+$ and $\pi^{\alpha}_-$ where $\alpha \in \core(T)$.

In general a combinatorial embedding $\pi^G$ gives $ \acc(G,\pi^{G})=1$ if and only if $\pi^{\alpha} = \pi^{\alpha}_+$ for all $\alpha \in \core(T)$, or $\pi^{\alpha} = \pi^{\alpha}_-$ for all $\alpha \in \core(T)$. For the vertices $\beta \in V(T) \setminus \core(T)$, in which case $G_{\beta}$ is always finite, the choice of $\pi^{\beta}$ does not influence  $\acc (G,\pi^G)=1$.

\begin{example}
Consider the $2$-connected graph shown in Figure \ref{fig:2-connected_example}. Lemma \ref{lemma:full_core} implies that the combinatorial embedding (with $\acc(G)=1$) is essentially unique for the core. On the other hand at every red component we have a choice between two possible ways of embedding, see Figure \ref{fig:2-connected_example_revisited}. 
\end{example}

\begin{figure}[h]
     \centering
         \begin{tikzpicture}[node distance=1cm,
main node/.style={circle,minimum width=1.9pt, fill, inner sep=0pt,outer sep = 0pt},
red node/.style={circle,minimum width=2.2pt, fill,red, inner sep=0pt,outer sep = 0pt},
blue node/.style={circle,minimum width=2.2pt, fill,blue, inner sep=0pt,outer sep = 0pt},
square node/.style={draw,regular polygon,regular polygon sides=4,fill,inner sep=0.9pt,outer sep=0.1pt},
triangle node/.style={draw,regular polygon,regular polygon sides=3,fill,inner sep=0.8pt,outer sep=0.1pt}]   

  \node[main node] (A1) at (0,0){};
  \node[main node] (A2) at (0,1){};
  \node[main node] (A3) at (1,0){};
  \node[main node] (A4) at (1,1){};

  \node[main node] (B1) at (2.25,0.25){};
  \node[main node] (B2) at (2.25,0.75){};
  \node[main node] (B3) at (2.75,0.25){};
  \node[main node] (B4) at (2.75,0.75){};
  
  \node[main node] (C1) at (0.25,2.25){};
  \node[main node] (C2) at (0.25,2.75){};
  \node[main node] (C3) at (0.75,2.25){};
  \node[main node] (C4) at (0.75,2.75){};
  
  \node[main node] (D1) at (-1.75,0.25){};
  \node[main node] (D2) at (-1.75,0.75){};
  \node[main node] (D3) at (-1.25,0.25){};
  \node[main node] (D4) at (-1.25,0.75){};  

  \node[main node] (E1) at (0.25,-1.75){};
  \node[main node] (E2) at (0.25,-1.25){};
  \node[main node] (E3) at (0.75,-1.75){};
  \node[main node] (E4) at (0.75,-1.25){};

  \path [draw=gray, thin] (A1) edge (A2);
  \path [draw=gray, thin] (A1) edge (A3);
  \path [draw=gray, thin] (A4) edge (A2);  
  \path [draw=gray, thin] (A4) edge (A3);
  
  \path [draw=gray, thin] (B1) edge (B2);
  \path [draw=gray, thin] (B1) edge (B3);
  \path [draw=gray, thin] (B4) edge (B2);  
  \path [draw=gray, thin] (B4) edge (B3);
  
  \path [draw=gray, thin] (C1) edge (C2);
  \path [draw=gray, thin] (C1) edge (C3);
  \path [draw=gray, thin] (C4) edge (C2);  
  \path [draw=gray, thin] (C4) edge (C3);

  \path [draw=gray, thin] (D1) edge (D2);
  \path [draw=gray, thin] (D1) edge (D3);
  \path [draw=gray, thin] (D4) edge (D2);  
  \path [draw=gray, thin] (D4) edge (D3);

  \path [draw=gray, thin] (E1) edge (E2);
  \path [draw=gray, thin] (E1) edge (E3);
  \path [draw=gray, thin] (E4) edge (E2);  
  \path [draw=gray, thin] (E4) edge (E3);

  \path [draw=gray, thin] (A1) edge (E2);
  \path [draw=gray, thin] (A3) edge (E4);

  \path [draw=gray, thin] (A1) edge (D3);
  \path [draw=gray, thin] (A2) edge (D4);

  \path [draw=gray, thin] (A2) edge (C1);
  \path [draw=gray, thin] (A4) edge (C3);

  \path [draw=gray, thin] (A4) edge (B2);
  \path [draw=gray, thin] (A3) edge (B1);
  
  \node[red node] (X1) at (0.6,0.7){};
  \node[red node] (X2) at (0.6,0.3){};
  \node[red node] (X1a) at (1.4,0.7){};
  \node[red node] (X2a) at (1.4,0.3){};  
  \node[red node] (Y) at (-0.5,1.1){};
  \node[red node] (Ya) at (-0.5,0.7){};
  \node[red node] (Z) at (0.5,-0.5){};
  \node[red node] (Za) at (-0.4,-0.5){};

  \path [draw=red, thin] (A4) edge (X1);
  \path [draw=red, thin] (X2) edge (X1);  
  \path [draw=red, thin] (A3) edge (X2);  
  \path [draw=red, thin] (A2) edge (Y);
  \path [draw=red, thin] (D4) edge (Y);
  \path [draw=red, thin] (A1) edge (Z);
  \path [draw=red, thin] (E2) edge (Z); 
  
  \path [draw=red, thin, densely dashed] (A4) edge (X1a);
  \path [draw=red, thin, densely dashed] (X2a) edge (X1a);  
  \path [draw=red, thin, densely dashed] (A3) edge (X2a);  
  \path [draw=red, thin, densely dashed] (A2) edge (Ya);
  \path [draw=red, thin, densely dashed] (D4) edge (Ya);
  \path [draw=red, thin, densely dashed] (A1) edge (Za);
  \path [draw=red, thin, densely dashed] (E2) edge (Za);

  \node (E) at (-0.25,-1.5)[label=center:{$\cdots$}]{};
  \node (E') at (0.5,-2.2)[label=center:{$\vdots$}]{};
  \node (E'') at (1.3,-1.5)[label=center:{$\cdots$}]{};

  \node (C) at (-0.25, 2.5)[label=center:{$\cdots$}]{};
  \node (C') at (0.5, 3.4)[label=center:{$\vdots$}]{};
  \node (C'') at (1.3, 2.5)[label=center:{$\cdots$}]{};

  \node (B) at (3.35, 0.5)[label=center:{$\cdots$}]{};
  \node (B') at (2.5,1.4)[label=center:{$\vdots$}]{};
  \node (B'') at (2.5,-0.2)[label=center:{$\vdots$}]{};
  
  \node (D) at (-2.25, 0.5)[label=center:{$\cdots$}]{};
  \node (D') at (-1.5,1.4)[label=center:{$\vdots$}]{};
  \node (D'') at (-1.5,-0.2)[label=center:{$\vdots$}]{};  

\end{tikzpicture}
        \caption{A $2$-connected graph with infinitely many ends and a simply connected embedding}
        \label{fig:2-connected_example_revisited}
\end{figure}

\begin{remark}
Notice that when $G$ has $1$ end we have $\core(T)=\emptyset$ or it is a single point. And indeed, any choice of combinatorial embeddings for the Tutte components results in a combinatorial embedding of $G$ with $ \acc(G)= 1$.
\end{remark}

\subsection{2-connected graphs with special vertices} \label{subsection:technical_difficulty}
When we later consider arbitrary planar graphs decomposed into 2-connected blocks, there will be a need to treat certain special vertices of $G$, informally speaking, as if they were ends of $G$. That is, as if they alone counted as infinitely many vertices, and by themselves gave rise to an accumulation point in one of the incident faces. The reason we did not include this technicality so far is that Lemma \ref{lemma:full_core} - which is the key ingredient of our construction - is simpler to present without it, and the proof of Lemma \ref{lemma:full_core} is heavy in notation as it is. Moreover, the need for these special vertices arises further down the line, and their presence does not change the proof, only complicates the exposition.

Nevertheless, to be precise, we now explain how the definition of $\core(G)$, $\acc(G)$, and the statement and proof of Lemma \ref{lemma:full_core} need to be modified in order to incorporate the special vertices. See the upcoming subsection \ref{subsection:example_spec_vertices} for an example where the special vertices are crucial.

Assume that a set of vertices $W \subseteq V(G)$ is specified to be {\it special}. 
Earlier $V_{\infty}$ stood for the set of Tutte components $\alpha$ where $|G_{\alpha}|=\infty$, now we redefine $V_{\infty}$ so that every $\alpha$ with $V(G_{\alpha}) \cap W \neq \emptyset$ is also included.
Recall that $\core(G)$ is defined as $\Gamma(\core(T))$, where $\core(T)$ is the convex hull of $V_{\infty} \cup \Ends(T)$ in the Tutte tree $T$. This definition remains the same, but with the updated $V_{\infty}$. Note that we automatically have $W \subseteq \core(G)$.

The definition of $\acc(G, \pi)$ changes as well. When considering a finite subgraph $H \subset G$, $a^{\pi}_H$ denoted the number of combinatorial faces of $H$ which surrounded infinitely many vertices of $G \setminus H$. Now we update that definition by counting combinatorial faces surrounding vertices from $W$ as well. (Note that only vertices from $W\setminus V(H)$ can be surrounded by combinatorial faces of $H$.) We set $\acc(G,\pi) = \sup a^{\pi}_{H_i}$ as before, with $H_1 \subset H_2 \subset\ldots \subset G$ an exhaustion by finite graphs as before. 

Earlier in Lemma \ref{lemma:full_core} we assumed that $G$ has a simply connected embedding, which means a unique accumulation point $x \in S^2$. Let $D$ denote the connected component of $x$ in $S^2 \setminus \iota(G)$. In the plane $D \setminus  \{x\}$ is the infinite topological face of the graph (which might be empty).
When taking the special vertices $W$ into consideration, we need a stronger assumption in Lemma \ref{lemma:full_core}, namely that there exists a simply connected embedding such that vertices in $W$ are embedded on the topological boundary of $D$. The implication remains that $G$ has a unique (up to inverting all permutations) combinatorial embedding with $\acc(G)=1$.

\begin{lemma} \label{lemma:full_core_special_vertices}
Let $G$ be an infinite, $2$-connected planar graph with special vertices $W \subseteq V(G)$. Assume that $\core(G) = G$, and that $G$ has a simply connected embedding into $S^2$ which maps $W$ onto $\partial D$. Then (up to inverting all permutations) $G$ has a unique combinatorial embedding with $\acc(G)=1$.
\end{lemma}
 
The proof of Lemma \ref{lemma:full_core_special_vertices} is essentially the same as the proof of Lemma \ref{lemma:full_core}. Earlier, contradictions arose because we could establish $\acc(G,\pi) \geq 2$ by finding cycles with either infinitely many vertices, or virtual edges representing infinitely many vertices both on the inside and the outside of the cycle. When the special vertices $W$ are present, finding such a vertex inside or outside the cycle also suffices to show $\acc(G,\pi) \geq 2$.  
  
\subsection{An example with special vertices} \label{subsection:example_spec_vertices}

Consider the graph $G$ in Figure \ref{fig:example_spec_vertices}. The vertex $v$ is a cut vertex, and let $G_a$ denote the infinite $2$-connected component of $G$. We aim to illustrate that $v$ needs to be treated as a special vertex in $G_a$ in order to construct a combinatorial embedding of $G$ with $\acc(G)=1$.

Let us first explore what happens if we do not treat $v$ as special. Let $C$ denote the 4-cycle in $G_a$ containing $v$, and $P$ the infinite path that $v$ separates from $G_a$. Notice, that $C$ is not in $\core(G_a)$, in fact $\core(G_a)$ is the infinite ladder from Figure \ref{fig:tutte_decomp_ladder}. Lemma \ref{lemma:full_core} says that $\core(G_a)$ has an essentially unique combinatorial embedding with $\acc(\core(G_a))=1$ (which corresponds to the way it is drawn in these pictures, the only accumulation point being at infinity). In the Tutte decomposition of $G_a$, the cycle $C$ is amalgamated to the cycle $C'$ in $\core(G_a)$ through a 3-link $L$. When we will later choose a random combinatorial embedding of $G_a$ as described in subsection \ref{subsection:random_infinite_block}, the combinatorial embeddings of $C$ and $L$ are going to be chosen uniformly randomly, because they are not in the core. The cycle $C$ has only 1, but $L$ has 2 choices. And this choice of $\pi^L$ determines weather $C$ will be embedded inside $C'$, or outside it. In case $C$ is mapped inside $C'$ we will get $\acc(G)=2$ no matter how we reattach the infinite path $P$. So we will not manage to build the right combinatorial embedding this way.

Now let us explore what happens if we treat $v$ as special in $G_a$. In this case $C$ is also included in $\core(G_a)$, so $G_a=\core(G_a)$. We see that the embedding places $v$ on the boundary of the infinite face, so Lemma \ref{lemma:full_core_special_vertices} applies. It provides us with the essentially unique way of combinatorially embedding $G_a$ such that we can then reattach $P$ while keeping $\acc(G)=1$.

\begin{figure}[h]
     \centering
         \begin{tikzpicture}[node distance=1cm,
main node/.style={circle,minimum width=1.9pt, fill, inner sep=0pt,outer sep = 0pt},
red node/.style={circle,minimum width=2.2pt, fill,red, inner sep=0pt,outer sep = 0pt},
blue node/.style={circle,minimum width=2.2pt, fill,blue, inner sep=0pt,outer sep = 0pt},
square node/.style={draw,regular polygon,regular polygon sides=4,fill,inner sep=0.9pt,outer sep=0.1pt},
triangle node/.style={draw,regular polygon,regular polygon sides=3,fill,inner sep=0.8pt,outer sep=0.1pt}]   

  \node (d) at (-3,0)[label=center:{$\cdots$}]{};
  \node (e) at (3,0)[label=center:{$\cdots$}]{};
  \node[main node] (a) at (-1.4,0.7){};
  \node[main node] (b) at (0,0.7){};
  \node[main node] (c) at (1.4,0.7){};
  \node[main node] (u) at (-1.4,-0.7){};
  \node[main node] (v) at (0,-0.7){};
  \node[main node] (w) at (1.4,-0.7){};
  \node[main node] (x) at (-1.4,2.1){};
  \node[main node] (y) at (0,2.1)[label=right:{$v$}]{}; 
  \node[main node] (y1) at (0.5,2.6){};
  \node[main node] (y2) at (1,2.6){};
  \node[main node] (y3) at (1.5,2.6){};
  \node (y4) at (2,2.6)[label=center:{$\dots$}]{};  
  \node (C) at (-0.7,1.4)[label=center:{$C$}]{};
  \node (C') at (-0.7,0)[label=center:{$C'$}]{};  
  \node (P) at (1.5,2.9)[label=center:{$P$}]{};

  \path [draw=gray, thin] (y) edge (y1);
  \path [draw=gray, thin] (y1) edge (y2);
  \path [draw=gray, thin] (y2) edge (y3);

  \path [draw=gray, thin] (a) edge (b);
  \path [draw=gray, thin] (b) edge (c);
  \path [draw=gray, thin] (a) edge (u);
  \path [draw=gray, thin] (b) edge (v);
  \path [draw=gray, thin] (c) edge (w);
  \path [draw=gray, thin] (u) edge (v);
  \path [draw=gray, thin] (v) edge (w);
  \path [draw=gray, thin] (a) edge (x);
  \path [draw=gray, thin] (b) edge (y);
  \path [draw=gray, thin] (x) edge (y);

  \path [draw=gray, thin] (-2.8,0.7) edge (a);
  \path [draw=gray, thin] (-2.8,-0.7) edge (u);
  \path [draw=gray, thin] (c) edge (2.8,0.7);
  \path [draw=gray, thin] (w) edge (2.8,-0.7);

		\end{tikzpicture}
		\caption{An example with special vertices}
		\label{fig:example_spec_vertices}

\end{figure}

\subsection{Towards combinatorial embeddings of arbitrary con\-nec\-ted planar graphs}

Having found (essentially unique) combinatorial embeddings with $1$ combinatorial accumulation point for $2$-connected graphs, we now turn our attention to general connected graphs.

Let $G$ be any locally finite, connected planar graph, with block-cut tree $\mathcal{T}$. We write $G_a$, $a \in A$ for the 2-connected components and $C$ for the cut vertices of $G$.

A combinatorial embedding $\pi^G$ induces a combinatorial embedding $\pi^{G_a}$ on all blocks. In order to be able to reconstruct $\pi^{G}$ we also have to record how the $\pi^{G_a}$ are to be glued at the cut vertices. 

For a block $G_a$ we denote by $\Cut(G_a)$ the set of vertices $C \cap V(G_a)$. Similarly, we write $\Cut_{\infty}(G_a)$ for those vertices in $\Cut(G_a)$ that separate $G_a$ from infinitely many points in $G$. 

\[\Cut_{\infty}(G_a)= \{v \in V(G_a) \mid |\mathcal{C}_{G \setminus E(G_a)}(v)| = \infty\}.\]

Our aim is to construct $\pi^G$ randomly such that $\acc(G, \pi^G)=1$. We will do this by first choosing the $\pi^{G_a}$ randomly for all blocks $G_a$, using Subsection \ref{subsection:general_2-connected} in the infinite case. Secondly, when making choices on how to glue the $\pi^{G_a}$ at the cut vertices, we will make sure that in the end $\acc(G, \pi^G)=1$ holds. 

\subsection{Finite blocks with special vertices} \label{subsection:finite_block}

Let $G_a$ be a block, and assume first that $G_a$ is finite. Let $v\in \Cut_{\infty}(G_a)$, and let $G_{b_1}, \ldots, G_{b_l}$ denote the other blocks containing $v$. 
For any combinatorial embedding of $G$ each $G_{b_i}$ is surrounded by one of the combinatorial faces of $G_a$ that are incident to $v$. If $\mathcal{C}_{G \setminus G_a} (G_{b_i})$ is infinite, this produces an accumulation point inside that combinatorial face.

So in order to achieve $\acc(G,\pi^{G})=1$ we need all vertices in $\Cut_{\infty}(G_a)$ to be adjacent to the same combinatorial face of $G_a$. Moreover, all neighboring blocks $G_b$ with $\mathcal{C}_{G \setminus G_a} (G_{b_i})$ infinite have to fall in this distinguished combinatorial face. 

\begin{remark}
Note that fixing $\pi^{G_a}$ for each $a \in A$, and also fixing the choice of the combinatorial face of each $(G_a , \pi^{G_a})$ where each neighboring $G_b$ should be embedded still does not determine the global combinatorial embedding $\pi^G$. When there are several blocks joined by a single cut-vertex we still need to choose a cyclic order of the blocks. This choice however does not impact $\acc(G, \pi^G)=1$.
\end{remark}

We choose a random combinatorial embedding $\pi^{G_a}$ such that $\pi^{G_a}$ places all of $\Cut_{\infty}(G_a)$ on the same distinguished face as follows. The simply connected embedding of $G$ (that we assumed it has) induces such a combinatorial embedding on $G_a$. We wire $\Cut_{\infty}(G_a)$ together. That is, add an auxiliary vertex $v_a$ to $V(G_a)$ and connect $v_a$ with all $v \in \Cut_{\infty}(G_a)$. This graph $G_a^w$ is still planar, and as it is finite it has finitely many combinatorial embeddings. We choose a (uniform) random combinatorial embedding of $G_a^w$, and choose the $G_a$-face that surrounds $v_a$ to be distinguished. 

\subsection{Infinite blocks with special vertices} \label{subsection:infinite_blocks}
For an infinite block $G_a$ the role of the distinguished combinatorial face will be played by the ``infinite face'' of $G_a$. We will not define the infinite face, but rather say that a point ``sees infinity''. We will then make sure that $\acc(G_a,\pi^{G_a})=1$ and all of $\Cut_{\infty}(G_a)$ can see infinity. 

Since $G_a$ is infinite it is no longer true that any dart $\righte$ is part of a unique finite combinatorial face. Let $W(\righte)$ denote the unique bi-infinite walk  containing $\righte$ with the property that  at every vertex the next step uses the dart succeeding the opposite of the previous. When $G_a$ was finite, $W(\righte)$ kept looping the finite combinatorial face containing $\righte$ infinitely many times in both directions. When $G_a$ is infinite however, $W(\righte)$ might be a bi-infinite path of darts. We say a vertex $x \in V(G_a)$ \emph{sees infinity} in the combinatorial embedding $\pi^{G_a}$ if there is a dart $\righte$ with $t(\righte)=x$ such that $W(\righte)$ is bi-infinite. Let $\rightf$ denote the dart after $\righte$ in $W(\righte)$, i.e.\ $\rightf=\pi^{G_a}_x(\lefte)$. (Recall that $\pi_x^{G_a}$ denotes the cyclic permutation on the outgoing darts at the vertex $x$ in the combinatorial embedding $\pi^{G_a}$.)

We call the pair $(\lefte, \rightf)$ the \emph{infinite region} at $x$. We claim that this pair, if it exists, is unique for $x$. Indeed, as $G_a$ is 2-connected, we can find a directed path $\overrightarrow{P}$ from $t(\rightf)$ to $s(\righte)$, and form the finite cycle $\overrightarrow{C}=(\overrightarrow{P},\righte, \rightf)$. The walk $W(\righte)$ has infinitely many points inside $\overrightarrow{C}$. If there was another pair $\lefte_1, \rightf_1$ giving an infinite region at $x$, then $W(\righte_1)$ would have infinitely many points outside $\overrightarrow{C}$ contradicting $\acc(G_a,\pi^{G_a})=1$.

\begin{lemma} \label{lemma:infinite_face}
Let $G_a$ be an infinite block with a fixed combinatorial embedding $\pi^{G_a}$ such that $\acc(G_a, \pi^{G_a}) = 1 $. A vertex $x \in V(G_a)$ sees infinity if and only if there is no cycle separating it from infinitely many vertices of $G_a$.
\end{lemma}

\begin{proof}
First assume towards contradiction that $x$ sees infinity and there is a cycle $\overrightarrow{C}$ surrounding $x$ with infinitely many points on the outside. The bi-infinite path $W(\righte)$ cannot cross $\overrightarrow{C}$, so there are infinitely many vertices inside $\overrightarrow{C}$ as well, contradicting $\acc(G_a,\pi^{G_a})=1$.

On the other hand assume that $x$ does not see infinity, that is all darts leaving $x$ are part of finite combinatorial faces. Form the finite subgraph $H$ consisting of the vertices and edges of these combinatorial faces, together with the combinatorial embedding $\pi^H$ restricted from $\pi^{G_a}$. This is a finite graph, and none of its combinatorial faces adjacent to $x$ surround any vertices of $G_a$ (when considered as oriented cycles in $G_a$). However, all vertices in $V(G_a) \setminus V(H)$ are surrounded by some combinatorial face of $H$. As $G_a$ is infinite, there is some other combinatorial face of $H$, not adjacent to $x$ that surrounds infinitely many points of $G_a$. This, as a cycle in $G_a$, separates $x$ from infinitely many vertices of $G_a$. 
\end{proof}

\subsection{Random combinatorial embeddings of infinite blocks} \label{subsection:random_infinite_block}

We now construct a random combinatorial embedding of an infinite block $G_a$ using Lemma \ref{lemma:full_core_special_vertices}, by setting $W = \Cut_{\infty}(G_a)$.  

If $G$ has at least 2 ends, then any infinite $G_a$ has at least 2 ends (where vertices from $\Cut_{\infty}(G_a)$ are also regarded as ``ends''), so $\core(G_a)$ is not empty. Also, if $G$ has a simply connected embedding $\iota$, then the restriction $\iota|_{G_a}$ is a simply connected embedding of $G_a$ placing $\Cut_{\infty}(G_a)$ on the topological boundary of the infinite topological face of $G_a$. Thus the block $G_a$ satisfies the assumptions of Lemma \ref{lemma:full_core_special_vertices}.

In subsection \ref{subsection:technical_difficulty} we emphasized that special vertices always belong to the core, so $\Cut_{\infty}(G_a) \subseteq V(\core(G_a))$. We claim that with respect to either of the two combinatorial embeddings $\pi^{\core(G_a)}_+$ and $\pi^{\core(G_a)}_-$, all vertices $x \in \Cut_{\infty}(G_a)$ have to see infinity in $\core(G_a)$ (where $\pi^{\core(G_a)}_+$ and $\pi^{\core(G_a)}_-$ stand for the unique combinatorial embedding and its inverse, as in Lemma \ref{lemma:full_core}). Indeed, by Lemma \ref{lemma:infinite_face}, if $x$ did not see infinity, it would be separated from infinitely many points by a cycle, witnessing $ \acc(\core(G)) \geq 2$, using that $x$ is special.

\begin{remark}
If some $G_\alpha$ is 2-connected and has no vertices that see infinity then we see that it 
cannot contain $\Cut_{\infty}$-vertices at all, since that would give rise to at least 2 accumulation points of the whole graph $G$.
\end{remark}

Note that amalgamating the remaining Tutte components $\alpha \notin \core(T_a)$ to $\core(G_a)$ cannot create new cycles separating vertices of $\Cut_{\infty}(G_a)$ from infinitely many points. So the choices of the $\pi^{\alpha}$ for $\alpha \notin \core(T_a)$ are arbitrary, as before in subsection \ref{subsection:general_2-connected}. These choices do not influence $\Cut_{\infty}(G_a)$ seeing infinity. 

We make independent uniform random choices between $\pi^{\core(G_a)}_+$ and $\pi^{\core(G_a)}_-$ for the core, and also to choose each $\pi^{\alpha}$  for $\alpha \notin \core(T_a)$. The chosen $\pi^{\core(G_a)}$ and $\pi^{\alpha}$ together form the random combinatorial embedding $\pi^{G_a}$.

\subsection{Random embedding of $G$ from its embedded blocks}

\label{subsection:general_connected}

Recall that at this point $G$ is a deterministic, infinite, locally finite graph that has a locally finite embedding. We now construct a random combinatorial embedding $\pi^G$ of $G$ with $\acc(G,\pi^G) = 1$. 

For finite blocks $G_a$, as described in subsection \ref{subsection:finite_block}, we choose a uniform random combinatorial embedding $\pi^{G_a}$ with the additional property that the vertices in $\Cut_{\infty}(G_a)$ are incident to a combinatorial face that we consider distinguished.

For infinite blocks $G_a$, we chose $\pi^{G_a}$ as in subsection \ref{subsection:random_infinite_block}, with $\Cut_{\infty}(G_a)$ considered as special points representing infinitely many vertices (as explained in subsection \ref{subsection:technical_difficulty}, and therefore being placed so that they see infinity). 

We then construct the combinatorial embedding $\pi^G$ by putting together the $\pi^{G_a}$. At cut vertices $v \in \Cut_{\infty}(G_a)$ we define $\pi^{G}_v$ so that all edges leading to all other blocks $G_b$ intersecting $G_a$ at $v$ are placed inside the distinguished face of $(G_a, \pi^{G_a})$. In the infinite case this means that in the cyclic order we place these edges between the two darts $(\lefte, \rightf)$ that form the unique infinite region of $G_a$ at $v$. 

There is some freedom still in the choice of $\pi_v$ when more than two blocks are glued at $v$, see the Remark in subsection \ref{subsection:finite_block}. In that case we choose uniformly randomly among the finitely many choices satisfying the above.

We claim that $\acc(G,\pi^G) = 1$. Indeed, any cycle $C$ has to belong to a block $G_a$. The choice of the distinguished face (or the infinite region in case $G_a$ is infinite) makes sure that $C$ cannot have infinitely many points on both sides.

\begin{remark}
Note that so far in the present Section \ref{section:unimod_combin} the graph $G$ was deterministic and unrooted (but assumed to have a locally finite embedding).
\end{remark}

\subsection{Proof of Theorem \ref{theorem:unimodular_combinatorial_embedding_1_acc}}

The following proposition completes the proof of Theorem \ref{theorem:unimodular_combinatorial_embedding_1_acc}.

\begin{proposition}
Let $(G,o)$ be a URPG. Given a random instance of $(G,o)$, choose the combinatorial embedding $\pi$ of $G$ randomly as described in subsection \ref{subsection:general_connected}. Then the random triple $(G,o;\pi)$ is a unimodular combinatorial embedding of $(G,o)$.
\end{proposition}

\begin{proof}
Let $(G,o_1, o_2; \pi)$ denote the birooted (decorated) random graph  induced from $(G, o; \pi)$ by taking a random step in the degree-biased version of $(G, o; \pi)$. We have to check that the distribution of $(G, o_1, o_2; \pi)$ is invariant with respect to swapping the $o_1$ and $o_2$. 

First the unimodularity of $(G,o)$ says that $(G, o_1, o_2)$ is invariant under swapping $o_1$ and $o_2$. 
We have two operations on the random graph $(G,o)$: one is taking a random step on a degree-biased sample of $(G,o)$, the second is decorating $(G,o)$ randomly. These two operations commute, their order is irrelevant. The bias and the random step never considers the combinatorial embedding and the random combinatorial embedding never considers any roots.
\end{proof}

\section{Embeddings of URPG's into $\H^2$ and $\R^2$}\label{section:inv_uni_emb}

In this section we present the proof of Theorem \ref{thm:invariant_and_unimodular_embedding}. Much of it is adopted from \cite{benjamini2019invariant}, so for brevity we will refer to them as much as possible, and emphasize the contribution of the present work instead.

We tackle the case of more than one ends through Theorem \ref{theorem:unimodular_combinatorial_embedding_1_acc} and Proposition \ref{prop:triangularization}. Besides these, the case not covered by \cite{benjamini2019invariant}, namely an invariant embedding of a non-amenable URPG into $\H^2$ will be dealt with using Theorem \ref{theorem:invariant_from_unimodular}.

\smallskip

\begin{proof}[Proof of Proposition \ref{prop:triangularization}]
By Theorem \ref{theorem:unimodular_combinatorial_embedding_1_acc} we can start with a unimodular random combinatorial embedding $\pi$ of $G$ with $\acc(G, \pi) \leq 1$. Based on this we can construct the triangulation $(G^+,o^+)$ as in \cite[Theorem 2.2]{benjamini2019invariant}, the resulting $G^+$ has finite expected degree, and invariant amenability or non-amenability is preserved, and positive density of $G$ in $G^+$ is guaranteed. Moreover the construction provides a unimodular combinatorial embedding $\pi^+$ with $\acc(G^+, \pi^+)=1$, where the restriction of $\pi^+$ to $G\leq G^+$ is $\pi$.

We claim $G^+$ has one end. If  $G^+$ had more than one end, there would be a finite induced subgraph $F$ with $G^+ \setminus F$ having at least two infinite components. These infinite components would also have to be surrounded by distinct combinatorial faces of $F$, because $G^+$ is triangulated (no edge can be added without violating planarity). In fact all vertices of $G^+ \setminus F$ that are surrounded by the same combinatorial face of $F$ belong to the same connected component of $G^+\setminus F$. Consequently $F$ would have two distinct faces surrounding infinitely many points, implying $\acc(G^+, \pi^+) \geq 2$.
\end{proof}

\smallskip

\begin{proof}[Proof of Theorem \ref{thm:invariant_and_unimodular_embedding}]
For amenable URPG's one in fact first finds an invariant embedding of finite intensity, and takes the Palm version to produce an unimodular embedding. By Proposition \ref{prop:triangularization} one can find $(G,o)$ as a positive density subgraph of an amenable planar triangulation. By \cite[Theorem 4.2]{benjamini2019invariant} such a triangulation can be embedded into $\R^2$ invariantly with vertices mapped onto some invariant point process of finite intensity.

For non-amenable URPG's the construction works in the other direction. We construct a unimodular embedding of positive intensity, and use Theorem \ref{theorem:invariant_from_unimodular} to deduce the existence of an invariant embedding. The triangulation $(G^+,o^+)$ has a circle packing represetation on $\H^2$ which is unique up to isometries \cite{he1993fixed, angel2016unimodular}. Therefore the embedding defined by the circle packing (by connecting centers of tangent circles by straight line segments) is a well defined rooted drawing of $(G^+,o^+)$ (as defined in Subsection \ref{subsec:unimodular_embeddings}). As $(G^+,o^+, S)$ is unimodular, restricting the rooted drawing to $(S,o^+)$ on the event that $o^+ \in S$,  we get a unimodular embedding of $(G,o)$. Moreover, by a Mass Transport argument the intensity is simply multiplied by $\P[o^+ \in S]>0$ when restricting the embedding. The payment function is $f([\mathfrak{G}^+,\mathfrak{o}^+,\mathfrak{u},\mathfrak{S}]) = \mathbbm{1}_{\mathfrak{u} \in \mathfrak{S}} \cdot \lambda\big( \mathrm{Vor}(\mathfrak{G}^+,\mathfrak{o}^+) \cap \mathrm{Vor}(\mathfrak{S},\mathfrak{u})\big)$. (The reader who uses gothic letters as rarely as the authors may want to note that $\mathfrak{S}$ is the gothic version of the letter $S$.)

We now prove that this unimodular embedding has positive intensity. All edges of $(G^+,o^+)$ are mapped to geodesics, and the embedding gives a triangulation of $\H^2$. We consider the following factor allocation: take the barycentric subdivision of each triangle into 6 smaller triangles, and assign each to the one original vertex it is incident to. The number of triangles assigned to $o^+$ is $2 \deg_{G^+}(o^+)$, so the expected area of these pieces is at most $(\pi/2)\cdot 2\E[\deg_{G^+}(o^+)] < \infty$. Consequently the intensity of the embedding of $(G^+,o^+)$ is positive. As $S$ has positive density in $G^+$, the embedding restricted to $S$ also has positive intensity. Therefore we get an invariant embedding by Theorem \ref{theorem:invariant_from_unimodular}.


Finally, URPG's have no invariant or unimodular embeddings into the wrong space as proven in Propositions \ref{prop:no_invar_nonamen}, \ref{prop:no_invar_amen} and \ref{prop:pathologies} below.
\end{proof}

\section{Cases with no embedding because of isoperimetry}\label{section:noembedding}

The final pieces of the proof of Theorem \ref{thm:invariant_and_unimodular_embedding} are the following propositions.  

\begin{proposition}\label{prop:no_invar_nonamen}
A non-amenable URPG $G$ has no invariant embedding into $\R^2$. 
\end{proposition}

\begin{proof}[Proof as in Theorem 1.1 in \cite{benjamini2019invariant}]
Suppose that $G$ had an isometry-invariant embedding into $\R^2$ without accumulation points. Then one could use the invariant random partitions of $\R^2$ to $2^n \times 2^n$ squares to define a unimodular finite exhaustion of $G$. Thus $G$ is amenable, a contradiction.
\end{proof}

\begin{proposition}\label{prop:no_invar_amen}
An amenable URPG $G$ has no invariant embedding into $\H^2$.
\end{proposition}

\begin{proof}
Suppose that $G$ has an invariant embedding into $\H^2$. Take the Voronoi partition of $\H^2$ corresponding to the embedded vertices. 
For the random configuration $\omega$ and $x,y\in\H^2$, define $f(x,y;\omega)$ to be 1 if $x$ is in the Voronoi cell of an embedded vertex that is at distance $\leq 1$ from $y$, and for $A,B\subset\H^2$ Borel define 
\[\mu (A,B)=\E \left[\int_{x\in A} \int_{y\in B} f(x,y;\omega) dx dy\right].\]
The continuous version of the Mass Transport Principle (Theorem 5.2 in \cite{benjamini2001percolation}) shows that the expected area of the Voronoi cell of the origin has to be finite. 

Consider some unimodular finite exhaustion of $G$, i.e.\ some unimodular random $\big(G,o;(P_n)_{n\in \mathbb{N}}\big)$ with the following properties almost surely. Each $P_n$ is a partition of $V(G)$ into finite subsets, $P_{n+1}$ is a coarsening of $P_n$, and for every $u,v \in V(G)$ there exists some $n$ for which $u$ and $v$ are in the same part in $P_n$. The existence of such a unimodular finite exhaustion is equivalent to the amenability of $(G,o)$.

The $P_n$ can be used to define an invariant random sequence of coarser and coarser partitions $K_n$ of $\H^2$, by taking the unions of the Voronoi cells of vertices that belong to the same part in $P_n$. The parts in $K_n$ have finite area. Any two points of $\H^2$ end up in the same part of $K_n$ for high enough $n$, which contradicts the non-amenability of $\H^2$. (To see this using our setup, take an invariant random non-amenable tiling in $\H^2$ with a transitive underlying graph $\Gamma$. The above sequence of partitions would generate a unimodular finite exhaustion of $\Gamma$, a contradiction.)
\end{proof}

\smallskip

Propositions \ref{prop:no_invar_nonamen} and \ref{prop:no_invar_amen} covered invariant embeddings; the next observations are their counterparts addressing unimodular embeddings.

\begin{proposition} \label{prop:pathologies}
Let $(G,o)$ be a URPG. 
\begin{enumerate}[(1)]
\item \label{pelda1} If $(G,o)$ is amenable, then it has no unimodular embedding of positive intensity into $\H^2$. 

\item \label{pelda2} There exists an amenable $(G,o)$ with a unimodular embedding (of 0 intensity) into $\H^2$.

\item \label{pelda3} If $(G,o)$ is non-amenable then it has no unimodular embedding of positive intensity into $\R^2$.
\end{enumerate}
\end{proposition}
\noindent

\begin{proof}
If an embedding as in (\ref{pelda1}) or (\ref{pelda3}) existed, one could apply Theorem \ref{theorem:invariant_from_unimodular} to obtain an invariant embedding, contradicting Propositions \ref{prop:no_invar_amen} and \ref{prop:no_invar_nonamen}, respectively. 

Example 1.5 in \cite{benjamini2019invariant} proves (\ref{pelda2}). That is, consider a bi-infinite geodesic in $\H^2$ through the origin $0$, and take the points at integer distance from $0$ on this geodesic to be the embedded vertices.

Part (\ref{pelda3}) is part of Theorem 1.2 in \cite{benjamini2019invariant}.
\end{proof}

\bigskip

\section*{Appendix}

\begin{proof}[Proof of Theorem \ref{theorem:invariant_from_unimodular}]
We denote by $\nu_H$ the Haar measure on $\Gamma$, normalized such that $\nu_H\big(\Stab_{\Gamma}(0)\big)=1$. Any factor allocation scheme, in particular $\Psi$ can be used to sample the Palm version of an invariant random embedded graph $\mathfrak{G}$ with finite intensity as follows. First pick $\mathfrak{G}$ randomly, then let $\mathfrak{o} \in V(\mathfrak{G})$ be the vertex whose cell contains the origin, i.e.\ $0 \in \Psi_{V(\mathfrak{G})}(\mathfrak{o})$. Now pick a Haar-random element $\varphi \in \Gamma$ such that $\varphi.\mathfrak{o}=0$ (the set of such $\varphi$ is compact, so it has finite Haar measure), and let $\mathfrak{G}' = \varphi.\mathfrak{G}$. Let $\mathfrak{G}^*$ denote the Palm version of $\mathfrak{G}$. Then in distribution $\mathfrak{G}'$ is exactly $\mathfrak{G}^*$ biased by the area of the cell of the origin, i.e.\ $\lambda\big(\Psi_{V(\mathfrak{G}^*)}(0)\big)$. To be precise, if $\mu'$ and $\mu^*$ denote the distributions of $\mathfrak{G}'$ and $\mathfrak{G}^*$ respectively, for any $A \subseteq \mathbb{G}_0(M)$ measurable we have
\[\mu'(A) = \int_{A} \lambda\big(\Psi_{V(\mathtt{G})}(0)\big)\ d \mu^*(\mathtt{G}) \bigg/ \int_{\mathbb{G}_0(M)} \lambda\big(\Psi_{V(\mathtt{G})}(0)\big) \ d \mu^*(\mathtt{G}).\]
For a reference see \cite{holroyd2005extra}, where a balanced allocation scheme (assigning equal sized cells for all centers) is constructed. In that case there is no bias happening, there is no difference between $\mathfrak{G}'$ and $\mathfrak{G}^*$.

We will reverse these steps by starting from a unimodular embedding $[\mathfrak{G_0}, \mathfrak{o}]$, sample it biased by $\lambda\big(\Psi_{V(\mathfrak{G_0})}(\mathfrak{o})\big)$, taking a representative of the class $[\mathfrak{G_0}, \mathfrak{o}]$ from $\mathbb{G}_0(M)$ with the root at $0$, picking a uniform random point $v$ from $\Psi_{V(\mathfrak{G_0})}(0)$ (the cell has finite area almost surely), picking a random isometry $\varphi \in \Gamma$ with $\varphi.0=v$, and shifting the embedding with $\varphi^{-1}$. We will show that the resulting embedding is $\Gamma$-invariant in distribution. This will complete the proof of Theorem \ref{theorem:invariant_from_unimodular}, as the original unimodular embedding is the Palm version of the invariant one, because our construction is the reverse of the one in the previous paragraph.

Throughout this paper this is the only point where fixing an arbitrary embedding of each rooted drawing $[\mathtt{G},\mathtt{o}]$ is necessary. As we have said before, $\RD(M)$ is in bijection with $\mathbb{G}_0(M)/\Stab_{\Gamma}(0)$, so let $\theta:\RD(M) \to \mathbb{G}_0(M)$ be a measurable selecting function (as both $\mathbb{G}_0(M)$ and $\RD(M)$ are standard Borel spaces, such a function exists, see for example \cite[Exercise 18.3]{kechris2012classical}). To shorten notation we write $\theta([\mathtt{G_0},\mathtt{o}]) = \overline{[\mathtt{G},\mathtt{o}]}$. Note that $\overline{[\mathtt{G},\mathtt{o}]}$ is no longer a rooted object, it is an embedded graph that happens to have a vertex at $0 \in M$.

Let $a: \mathbb{G}(M) \times M \to [0,\infty)$ be defined as follows:
\[a(\mathtt{G}, y ) = \begin{cases} 1 \textrm{ if } y \in V(\mathtt{G}) \textrm{ and } 0 \in \Psi_{V(\mathtt{G})}(y);\\ 0 \textrm{ otherwise.} \end{cases}\]
Note that if $0 \in M$ is not in the measure zero subset of $M$ 
where $\Psi_{V(\mathtt{G})}$ fails to be a partition, then it is contained in exactly one cell, so we have $\sum_{y \in V(\mathtt{G})} a(\mathtt{G}, y ) = 1$.

Let $\mu_0 \in \mathrm{Prob}\big(\RD(M)\big)$ denote the distribution of $(\mathfrak{G_0},\mathfrak{o})$. We define the measure $\mu$ on $\mathbb{G}(M)$ as follows.
\begin{equation}\label{eqn:def_of_mu}
\mu(A)=\int_{[\mathtt{G}, \mathfrak{o}] \in \RD(M)} \int_{\varphi \in \Gamma} \mathbbm{1}_A\left(\varphi.\overline{[\mathtt{G},\mathtt{o}]}\right) \cdot a\left(\varphi.\overline{[\mathtt{G},\mathtt{o}]}, \varphi.0 \right) \ d \nu_H(\varphi) \ d \mu_0([\mathtt{G},\mathtt{o}])
\end{equation}

\begin{remark}
In the definition of $\mu$ it is not apparent that any biasing with the size of the cell of $0$ took place. But in fact it did, otherwise there would be a $1/ \lambda\big(\Psi_{V\left(\overline{[\mathtt{G},\mathtt{o}]}\right)}(0)\big)$ term as well in the integral. The role of the biasing is exactly to cancel this term out. Also note that we are not normalizing $\mu$, as that would lengthen the expressions in the coming calculation. The measure of the whole space $\mu(\mathbb{G}(M))$ is $\mathbb{E}_{[\mathfrak{G_0},\mathfrak{o}]}\left[ \lambda \big( \Psi_{V(\mathfrak{G_0})}(\mathfrak{o}) \big) \right]$, this is ensured by having picked the right normalization of $\nu_H$. By assumption, this is finite, so we get a probability measure after normalization.
\end{remark}

We claim $\mu$ is $\Gamma$-invariant, which is equivalent to showing that for any $f: \mathbb{G}(M) \to [0,\infty)$ and fixed $\eta \in \Gamma$ we have
\begin{equation} \label{eqn:invariance}
\int_{\mathbb{G}(M)} \eta.f \ d \mu = \int_{\mathbb{G}(M)} f \ d \mu.
\end{equation}

By the definition of $\mu$ we have
\begin{eqnarray*}
\int_{\mathbb{G}(M)} \eta.f \ d \mu = \int_{\RD(M)} \int_{\Gamma} f\left(\eta^{-1}\varphi.\overline{[\mathtt{G},\mathtt{o}]}\right) \cdot a\left(\varphi.\overline{[\mathtt{G},\mathtt{o}]}, \varphi.0\right)\ d \nu_H(\varphi) \ d \mu_0([\mathtt{G},\mathtt{o}]) = \ldots
\end{eqnarray*}

For $\mu_0$-almost all $[\mathtt{G},\mathtt{o}] \in \RD(M)$ and $\varphi \in \Gamma$ we have $\sum_{x \in \eta^{-1}\varphi.V\left(\overline{[\mathtt{G},\mathtt{o}]}\right)} a\left(\eta^{-1}\varphi.\overline{[\mathtt{G},\mathtt{o}]}, x\right)$ = 1, so
\begin{eqnarray*}
\ldots = \int_{\RD(M)} \int_{\Gamma} \sum_{x \in \eta^{-1}\varphi.V\left(\overline{[\mathtt{G},\mathtt{o}]}\right)} a\left(\eta^{-1}\varphi.\overline{[\mathtt{G},\mathtt{o}]}, x\right) \cdot f\left(\eta^{-1}\varphi.\overline{[\mathtt{G},\mathtt{o}]}\right) \\ 
\cdot \ a\left(\varphi.\overline{[\mathtt{G},\mathtt{o}]}, \varphi.0\right)\ d \nu_H(\varphi) \ d \mu_0([\mathtt{G},\mathtt{o}])=\ldots
\end{eqnarray*}
By substituting $x = \eta^{-1}\varphi.y$ we have
\begin{eqnarray*}
\ldots = \int_{\RD(M)} \int_{\Gamma} \sum_{y \in V\left(\overline{[\mathtt{G},\mathtt{o}]}\right)} & a\left(\eta^{-1}\varphi.\overline{[\mathtt{G},\mathtt{o}]}, \eta^{-1}\varphi.y\right) \cdot f\left(\eta^{-1}\varphi.\overline{[\mathtt{G},\mathtt{o}]}\right) \\ & \cdot \ a\left(\varphi.\overline{[\mathtt{G},\mathtt{o}]}, \varphi.0\right)\ d \nu_H(\varphi) \ d \mu_0([\mathtt{G},\mathtt{o}])=\ldots
\end{eqnarray*}
We will use the Mass Transport Principle with the function $g$ that assigns to $[\mathtt{G},\mathtt{u}, \mathtt{v}] \in \mathbb{BRD}(M)$ the value \[g([\mathtt{G},\mathtt{u}, \mathtt{v}]) = a\left(\eta^{-1}\varphi.\overline{[\mathtt{G}, \mathtt{u}]}, \eta^{-1}\varphi.\overline{\mathtt{v}}\right) \cdot f\left(\eta^{-1}\varphi.\overline{[\mathtt{G}, \mathtt{u}]}\right) \cdot a\left(\varphi.\overline{[\mathtt{G}, \mathtt{u}]}, \varphi.0\right),\] where $\overline{\mathtt{v}} \in M$ is some point corresponding to $\mathtt{v}$ in the embedded $\overline{[\mathtt{G}, \mathtt{u}]}$, that is to say $\left[\overline{[\mathtt{G}, \mathtt{u}]}, 0 , \overline{\mathtt{v}}\right] = [\mathtt{G},\mathtt{u},\mathtt{v}]$. If there are multiple such points, any choice is sufficient, the value of $g$ does not depend on the choice. We get
\begin{align}
\nonumber \ldots &= \int_{\Gamma} \int_{\RD(M)}  \sum_{\mathtt{y} \in V\left(\mathtt{G}\right)} g([\mathtt{G},\mathtt{o}, \mathtt{y}]) \ d \mu_0([\mathtt{G},\mathtt{o}]) \ d \nu_H(\varphi)  = \\
\nonumber &\stackrel{\text{MTP}}{=} \int_{\Gamma} \int_{\RD(M)}  \sum_{\mathtt{y} \in V\left(\mathtt{G}\right)} g([\mathtt{G},\mathtt{y}, \mathtt{o}]) \ d \mu_0([\mathtt{G},\mathtt{o}]) \ d \nu_H(\varphi)  =\\
\label{eqn:after_MTP}
&=\int_{\RD(M)} \int_{\Gamma} \sum_{\mathtt{y} \in V\left(\mathtt{G}\right)} a\left(\eta^{-1}\varphi.\overline{[\mathtt{G}, \mathtt{y}]}, \eta^{-1}\varphi.\overline{\mathtt{o}}\right) \cdot f\left(\eta^{-1}\varphi.\overline{[\mathtt{G}, \mathtt{y}]}\right) \\ 
\nonumber & \quad \quad \quad \quad  \cdot \ a\left(\varphi.\overline{[\mathtt{G}, \mathtt{y}]}, \varphi.0\right)\ d \nu_H(\varphi) \ d \mu_0([\mathtt{G},\mathtt{o}])=\ldots
\end{align}



As we have emphasized before, the sum only makes sense once a representative of $[\mathtt{G},\mathtt{o}]$ is chosen. For us this will be $\overline{[\mathtt{G},\mathtt{o}]}$. This way $\overline{[\mathtt{G}, \mathtt{y}]}$ (as $\mathtt{y}$ runs through $V(\mathtt{G})$) is exactly $\overline{\left[\overline{[\mathtt{G}, \mathtt{o}]}, y \right]}$ as $y$ runs through $V(\overline{[\mathtt{G}, \mathtt{o}]})$. For an $y \in V\left(\overline{[\mathtt{G}, \mathtt{o}]}\right)$ let $\varphi_y \in \Gamma$ be such that $\varphi_y.\overline{\left[\overline{[\mathtt{G}, \mathtt{o}]}, y \right]}= \overline{[\mathtt{G}, \mathtt{o}]}$ and $\varphi_y.0=y$. We note that as we sum over $\mathtt{y}$, the term $\overline{\mathtt{o}}$ in (\ref{eqn:after_MTP}) also changes. As $y$ runs through $V\big(\overline{[\mathtt{G},\mathtt{o}]}\big)$ in our chosen representative 
$\varphi_y^{-1}.0$ is a valid choice of $\overline{\mathtt{o}}$. For $y$ fixed let $\psi=\eta^{-1}\varphi\varphi_y^{-1}$, and we calculate
\begin{align*}
\ldots &= \int_{\RD(M)} \int_{\Gamma} \sum_{y \in V\left(\overline{[\mathtt{G}, \mathtt{o}]}\right)} a\left(\eta^{-1}\varphi.\overline{\left[\overline{[\mathtt{G}, \mathtt{o}]}, y \right]}, \eta^{-1}\varphi\varphi_y^{-1}.0\right) \cdot f\left(\eta^{-1}\varphi.\overline{\left[\overline{[\mathtt{G}, \mathtt{o}]}, y \right]}\right) \\ 
& \quad \quad \quad \quad \quad \quad \cdot \ a\left(\varphi.\overline{\left[\overline{[\mathtt{G}, \mathtt{o}]}, y \right]}, \varphi.0\right)\ d \nu_H(\varphi) \ d \mu_0([\mathtt{G}, \mathtt{o}])=\\
&=\int_{\RD(M)} \sum_{y \in V\left(\overline{[\mathtt{G}, \mathtt{o}]}\right)} \int_{\Gamma}  a\left(\psi\varphi_y.\overline{\left[\overline{[\mathtt{G}, \mathtt{o}]}, y \right]}, \psi.0\right) \cdot f\left(\psi\varphi_y.\overline{\left[\overline{[\mathtt{G}, \mathtt{o}]}, y \right]}\right) \\
& \quad \quad \quad \quad \quad \quad \cdot \ a\left(\eta\psi\varphi_y.\overline{\left[\overline{[\mathtt{G}, \mathtt{o}]}, y \right]}, \eta\psi\varphi_y.0\right)\ d \nu_H(\varphi) \ d \mu_0([\mathtt{G}, \mathtt{o}])=\ldots
\end{align*}

The Haar measure $\nu_H$ is bi-invariant, therefore (for fixed $y$) we can integrate with respect to $\psi$ instead of $\varphi$, so we get
\begin{align*}
\ldots &= \int_{\RD(M)} \sum_{y \in V\left(\overline{[\mathtt{G}, \mathtt{o}]}\right)} \int_{\Gamma}  a\left(\psi.\overline{[\mathtt{G}, \mathtt{o}]}, \psi.0\right) \cdot f\left(\psi.\overline{[\mathtt{G}, \mathtt{o}]}\right) \\ 
&\quad \quad \quad \quad \quad \quad \cdot \ a\left(\eta\psi.\overline{[\mathtt{G}, \mathtt{o}]}, \eta\psi.y\right)\ d \nu_H(\psi) \ d \mu_0([\mathtt{G}, \mathtt{o}])= \\
&= \int_{\RD(M)} \int_{\Gamma} a\left(\psi.\overline{[\mathtt{G}, \mathtt{o}]}, \psi.0\right) \cdot f\left(\psi.\overline{[\mathtt{G}, \mathtt{o}]}\right) \\
&\quad \quad \quad \quad \quad \quad \cdot \sum_{y \in V\left(\overline{[\mathtt{G}, \mathtt{o}]}\right)} \ a\left(\eta\psi.\overline{[\mathtt{G}, \mathtt{o}]}, \eta\psi.y\right)\ d \nu_H(\psi) \ d \mu_0([\mathtt{G}, \mathtt{o}])=\ldots
\end{align*}
We have $\sum_{y \in V\left(\overline{[\mathtt{G}, \mathtt{o}]}\right)} \ a\left(\eta\psi.\overline{[\mathtt{G}, \mathtt{o}]}, \eta\psi.y\right)=1$ for almost all $\psi$ and $[\mathtt{G}, \mathtt{o}]$, therefore

\[\ldots = \int_{\RD(M)} \int_{\Gamma} a\left(\psi.\overline{[\mathtt{G}, \mathtt{o}]}, \psi.0\right) \cdot f\left(\psi.\overline{[\mathtt{G}, \mathtt{o}]}\right) \ d \nu_H(\psi) \ d \mu_0([\mathtt{G}, \mathtt{o}]) = \int_{\mathbb{G}(M)}f \ d \mu.\]

The last equality holds by the definition of $\mu$ in (\ref{eqn:def_of_mu}). Reading the calculation from start to finish yields equation (\ref{eqn:invariance}), proving that $\mu$ is indeed $\Gamma$-invariant, which completes the proof. 
\end{proof}

\begin{remark}
We never used that the function $a$ is $\{0,1\}$-valued, only that it sums to 1 almost always. For this reason one can extend Theorem \ref{theorem:invariant_from_unimodular} and Corollary \ref{cor:intensity} to \emph{fractional} allocations as well, where instead of partitioning $M$ we cover it by $L^{1}$ functions (indexed by the centers) that sum to $\mathbbm{1}_M$. The size of a cell is the integral of the function. Considering such fractional allocations has a utility in homogeneous spaces of general locally compact groups. In this generality the Voronoi allocation and balanced allocation of \cite{holroyd2005extra} can only be defined as fractional allocations, because the measure of points equidistant from two given points can be positive. 
\end{remark}

\noindent
\begin{acknowledgement}
The first author was partially supported by ERC Consolidator Grant 648017 and by Icelandic Research Fund Grant 185233-051. The second author was partially supported by NKFIH (National Research, Development and Innovation Office, Hungary) grant KKP-139502, ``Groups and graph limits''.
\end{acknowledgement}

\bibliographystyle{alpha}
\bibliography{refs}

\begin{thebibliography}{BWGT09}

\bibitem[AHNR16]{angel2016unimodular}
Omer Angel, Tom Hutchcroft, Asaf Nachmias, and Gourab Ray.
\newblock Unimodular hyperbolic triangulations: circle packing and random walk.
\newblock {\em Inventiones mathematicae}, 206(1):229--268, 2016.

\bibitem[AHNR18]{angel2018hyperbolic}
Omer Angel, Tom Hutchcroft, Asaf Nachmias, and Gourab Ray.
\newblock Hyperbolic and parabolic unimodular random maps.
\newblock {\em Geometric and Functional Analysis}, 28(4):879--942, 2018.

\bibitem[AL07]{aldous2007processes}
David Aldous and Russell Lyons.
\newblock Processes on unimodular random networks.
\newblock {\em Electronic Journal of Probability}, 12:1454--1508, 2007.

\bibitem[BHMK21]{baccelli2021unimodular}
Fran{\c{c}}ois Baccelli, Mir-Omid Haji-Mirsadeghi, and Ali Khezeli.
\newblock Unimodular {H}ausdorff and {M}inkowski dimensions.
\newblock {\em Electronic Journal of Probability}, 26:1--64, 2021.

\bibitem[BR03]{bonnington2003graphs}
C~Paul Bonnington and R~Bruce Richter.
\newblock Graphs embedded in the plane with a bounded number of accumulation
  points.
\newblock {\em Journal of Graph Theory}, 44(2):132--147, 2003.

\bibitem[BS01]{benjamini2001percolation}
Itai Benjamini and Oded Schramm.
\newblock Percolation in the hyperbolic plane.
\newblock {\em Journal of the American Mathematical Society}, 14(2):487--507,
  2001.

\bibitem[BT21]{benjamini2019invariant}
Itai Benjamini and Ádám Timár.
\newblock {Invariant embeddings of unimodular random planar graphs}.
\newblock {\em Electronic Journal of Probability}, 26:1--18, 2021.

\bibitem[BWGT09]{beineke2009topics}
Lowell~W Beineke, Robin~J Wilson, Jonathan~L Gross, and Thomas~W Tucker.
\newblock {\em Topics in topological graph theory}.
\newblock Cambridge University Press Cambridge, 2009.

\bibitem[CSKM13]{chiu2013stochastic}
Sung~Nok Chiu, Dietrich Stoyan, Wilfrid~S Kendall, and Joseph Mecke.
\newblock {\em Stochastic geometry and its applications}.
\newblock John Wiley \& Sons, 2013.

\bibitem[Die17]{diestel2017graph}
Reinhard Diestel.
\newblock {\em Graph theory}, volume 173 of {\em Graduate Texts in
  Mathematics}.
\newblock Springer-Verlag Berlin Heidelberg, 5th edition, 2017.

\bibitem[Hal66]{halin1966haufungspunktfreien}
R~Halin.
\newblock Zur h{\"a}ufungspunktfreien darstellung abz{\"a}hlbarer graphen in
  der ebene.
\newblock {\em Archiv der Mathematik}, 17(3):239--243, 1966.

\bibitem[HL05]{heveling2005characterization}
Matthias Heveling and G{\"u}nter Last.
\newblock Characterization of palm measures via bijective point-shifts.
\newblock {\em Annals of probability}, pages 1698--1715, 2005.

\bibitem[HP05]{holroyd2005extra}
Alexander~E Holroyd and Yuval Peres.
\newblock Extra heads and invariant allocations.
\newblock {\em Annals of probability}, 33(1):31--52, 2005.

\bibitem[HS93]{he1993fixed}
Zheng-Xu He and Oded Schramm.
\newblock Fixed points, koebe uniformization and circle packings.
\newblock {\em Annals of Mathematics}, pages 369--406, 1993.

\bibitem[Kec12]{kechris2012classical}
Alexander Kechris.
\newblock {\em Classical descriptive set theory}, volume 156.
\newblock Springer Science \& Business Media, 2012.

\bibitem[Las10]{last2010stationary}
G{\"u}nter Last.
\newblock Stationary random measures on homogeneous spaces.
\newblock {\em Journal of Theoretical Probability}, 23(2):478--497, 2010.

\bibitem[Lov12]{lovasz2012large}
L{\'a}szl{\'o} Lov{\'a}sz.
\newblock {\em Large networks and graph limits}, volume~60.
\newblock American Mathematical Soc., 2012.

\bibitem[LP17]{last2017lectures}
G{\"u}nter Last and Mathew Penrose.
\newblock {\em Lectures on the Poisson process}, volume~7.
\newblock Cambridge University Press, 2017.

\bibitem[LT09]{last2009invariant}
G{\"u}nter Last and Hermann Thorisson.
\newblock Invariant transports of stationary random measures and
  mass-stationarity.
\newblock {\em The Annals of Probability}, 37(2):790--813, 2009.

\bibitem[LZ13]{lando2013graphs}
Sergei~K. Lando and Alexander~K. Zvonkin.
\newblock {\em Graphs on surfaces and their applications}, volume 141.
\newblock Springer Science \& Business Media, 2013.

\bibitem[Mol12]{moller2012lectures}
Jesper Moller.
\newblock {\em Lectures on random Voronoi tessellations}, volume~87.
\newblock Springer Science \& Business Media, 2012.

\bibitem[MT01]{mohar2001graphs}
Bojan Mohar and Carsten Thomassen.
\newblock {\em Graphs on surfaces}.
\newblock Johns Hopkins University Press, 2001.

\bibitem[Nev77]{neveu1977processus}
Jacques Neveu.
\newblock Processus ponctuels.
\newblock In {\em Ecole d’Et{\'e} de Probabilit{\'e}s de Saint-Flour
  VI-1976}, pages 249--445. Springer, 1977.

\bibitem[RS04]{robertson2004graph}
Neil Robertson and Paul~D Seymour.
\newblock Graph minors. xx. wagner's conjecture.
\newblock {\em Journal of Combinatorial Theory, Series B}, 92(2):325--357,
  2004.

\bibitem[RZ90]{rother1990palm}
W~Rother and M~Z{\"a}hle.
\newblock Palm distributions in homogeneous spaces.
\newblock {\em Mathematische Nachrichten}, 149(1):255--263, 1990.

\bibitem[Tim21]{timar2017nonamenable}
Ádám Timár.
\newblock {A nonamenable “factor” of a Euclidean space}.
\newblock {\em The Annals of Probability}, 49(3):1427 -- 1449, 2021.

\bibitem[Tim23]{timar2023unimodular}
Ádám Timár.
\newblock Unimodular random one-ended planar graphs are sofic.
\newblock {\em Combinatorics, Probability and Computing}, page 1–8, 2023.

\end{thebibliography}

\bigskip
\noindent
{\bf \'Ad\'am Tim\'ar}\\
Division of Mathematics, The Science Institute, University of Iceland\\
and\\
Alfr\'ed R\'enyi Institute of Mathematics, Budapest\\
\texttt{madaramit[at]gmail.com}
\medskip
\ \\
{\bf L\'aszl\'o M\'arton T\'oth}\\
\'Ecole Polytechnique F\'ed\'erale de Lausanne\\
\texttt{laszlomarton.toth[at]epfl.ch}

\end{document}